\documentclass[11pt]{article}
\usepackage{a4wide}
\usepackage{amsmath}
\usepackage{amssymb}
\usepackage{amsthm}
\usepackage{amsfonts}
\usepackage{latexsym}
\usepackage{mathtools}
\usepackage{empheq}
\usepackage{bigints}
\usepackage[dvipdfmx]{graphicx}
\usepackage{comment}
\usepackage{cases}
\usepackage{float}
\usepackage{color}
\usepackage{titlefoot}
\usepackage{mathtools}
\usepackage{inputenc} 
\usepackage[T1]{fontenc}

\allowdisplaybreaks[1]

\newcommand{\dist}{\mathrm{dist}\,}          
\newcommand{\diver}{\mathrm{div}\,}		     
\newcommand{\sgn}{\mathrm{sgn}\,}            
\newcommand{\spt}{\mathrm{spt}\,}  			 

\newcommand{\capE}{\mathcal{E}}
\newcommand{\capF}{\mathcal{F}}

\newcommand{\capH}{\mathcal{H}}
\newcommand{\capI}{\mathcal{I}}

\newcommand{\mathN}{\mathbb{N}}

\newcommand{\mathR}{\mathbb{R}}
\newcommand{\mathS}{\mathbb{S}}

\let\OLDthebibliography\thebibliography
\renewcommand\thebibliography[1]{
	\OLDthebibliography{#1}
	\setlength{\parskip}{1.5pt}
	\setlength{\itemsep}{1pt plus 0.3ex}
}

\makeatletter

\@addtoreset{equation}{section}
\makeatother

\theoremstyle{mystyle}
\newtheorem{theorem}{Theorem}[section]
\newtheorem*{theorem*}{Theorem}
\newtheorem{lemma}[theorem]{Lemma}
\newtheorem{proposition}[theorem]{Proposition}
\newtheorem{corollary}[theorem]{Corollary}

\theoremstyle{definition}
\newtheorem{definition}[theorem]{Definition}
\theoremstyle{remark}

\begin{document}
\title{Local H\"older regularity of minimizers for nonlocal variational problems}

\author{
Matteo Novaga\footnote{Universit\`a di Pisa, Largo Bruno Pontecorvo 5, 56127 Pisa, Italy. E-mail: {\tt matteo.novaga@unipi.it}}\and
Fumihiko Onoue\footnote{Scuola Normale Superiore, Piazza dei Cavalieri 7, 56126 Pisa, Italy. E-mail: {\tt fumihiko.onoue@sns.it}}
}

\date{\today}	

\maketitle

\begin{abstract}
	We study the regularity of solutions to a nonlocal variational problem, which is related to the image denoising model, and we show that, in two dimensions, minimizers have the same H\"older regularity as the original image. More precisely, if the datum is (locally) $\beta$-H\"older continuous for some $\beta\in(1-s,\,1]$, where $s\in (0,1)$ is a parameter related to the nonlocal operator, we prove that the solution is also $\beta$-H\"older continuous.
\end{abstract}

\tableofcontents

\section{Introduction}
Let $K:\mathR^n\setminus\{0\} \to \mathR$ be a given function and $f:\mathR^n \to \mathR$ be a given datum. We study the minimization problem 
\begin{equation}\label{miniProb}
	\min\left\{ \capF_{K,f}(u) \mid u \in BV_K(\mathR^n) \cap L^2(\mathR^n)  \right\}
\end{equation}
where the functional $\capF_{K,f}$ is defined as
\begin{equation}\label{funcDenoisingProb}
	\capF_{K,f}(u) \coloneqq \frac{1}{2}\int_{\mathR^n}\int_{\mathR^n} K(x-y)\,|u(x) - u(y)|\,dx\,dy + \frac{1}{2}\int_{\mathR^n}(u(x)-f(x))^2\,dx
\end{equation} 
for any measurable function $u:\mathR^n \to \mathR$, and the space $BV_K(\mathR^n)$ is the set of all functions such that the first term of \eqref{funcDenoisingProb} is finite (see Section \ref{preliminaries} for the detail). The function $K$ is a kernel singular at the origin, and a typical example is the function $x \mapsto |x|^{-(n+s)}$ with $s \in (0,\,1)$. If $K$ is non-negative and we understand that $\capF_{K,\,f}(u) = +\infty$ when $u \in L^2(\mathR^n)\setminus BV_K(\mathR^n)$, then we observe that the functional $\capF_{K,\,f}$ is strictly convex, lower semi-continuous, and coercive in $L^2(\mathR^n)$. Hence, from the general theory of functional analysis (see, for instance, \cite{Brezis}), we obtain existence and uniqueness of solutions to \eqref{miniProb}. 

In this paper we focus on the specific kernel $K(x) = |x|^{-(n+s)}$, with $s\in (0,1)$, 
and we study the regularity of the minimizers of $\capF_{K,f}$, under some suitable conditions on the datum $f$. 
Our minimization problem is motivated by the classical variational problem
\begin{equation}\label{classicalDenosignProb}
	\min\left\{ \capF_{f}(u) \mid u \in BV(\mathR^n) \cap L^2(\mathR^n)  \right\}
\end{equation}
where $\capF_{f}(u)$ is defined as
\begin{equation}\label{classicalTotalVariFunc}
	\capF_{f}(u) \coloneqq \int_{\mathR^n}|\nabla u|\,dx + \frac{1}{2}\int_{\mathR^n} |u-f|^2\,dx.
\end{equation}
Th minimization problem \eqref{classicalDenosignProb} has been studied by many authors since the celebrated work by Rudin, Osher, and Fatemi \cite{ROF}, and plays an important role in image denoising and restoration (see for instance \cite{CCN, Brezis02}). From the perspective of image processing, the datum $f$ in the functional $\capF_{f}$ indicates an observed image and, when the given image has poor quality, then the minimizers of $\capF_{f}$ or solutions to the Euler-Lagrange equation associated with $\capF_{f}$ correspond to regularized images. It is easy to show that the minimizer of \eqref{classicalTotalVariFunc} exists and is unique, as a result of strict convexity, lower semi-continuity and coercivity of the functional. Moreover, the minimizer turns out to be the solution, in a suitable sense, of the Euler-Lagrange equation
\begin{equation}\label{classicalTotalVariEq}
	-\diver\left( \frac{\nabla u}{|\nabla u|}\right) + u - f = 0 \quad \text{ in } \mathR^n.
\end{equation}
The regularity of minimizers of $\capF_{f}$ have been studied by several authors. In particular, the global and local regularity was investigated in
a series of papers by Caselles, Chambolle and Novaga (see \cite{CCN01, CCN02, CCN}), who proved that the solution of \eqref{classicalTotalVariEq} inherits the local H\"older or Lipschitz regularity of the datum $f$, when the space-dimension $n$ is less than or equal to 7. In addition, if $f$ is globally H\"older or Lipschitz in a convex domain $\Omega \subset \mathR^n$, the global regularity also holds for the solution of \eqref{classicalTotalVariEq} with homogeneous Neumann boundary condition. 
In the recent papers \cite{Mercier, Porretta}, some of these results were extended to general dimensions. In \cite{Mercier}, Mercier has proved that the continuity of $f$ implies the continuity of a solution $u$ and, in the case of convex domains, the modulus of continuity is also inherited globally by the solution. 
Eventually, in \cite{Porretta}, Porretta was able to remove the restriction on the space-dimension. 

For the variational problems associated with the nonlocal total variation, Aubert and Kornprobst in \cite{AuKo} and Gilboa and Osher in \cite{GiOs, GiOs02} have proposed the methods for approximating the solutions to \eqref{classicalDenosignProb} with a sequence of nonlocal total variations associated with non-singular smooth kernels. However, as far as we know, there are no results on the regularity of minimizers of the functional $\capF_{K,f}$. Thus, in this paper, we consider the local H\"older regularity of the minimizers of \eqref{funcDenoisingProb} in two dimension as an analogy of the regularity results shown in \cite{CCN01, CCN02}. Precisely, we prove the following result:

\begin{theorem}\label{mainTheorem}
	Let $n=2$. Assume that $K(x) = |x|^{-(2+s)}$ with $s \in (0,\,1)$ and $f \in L^2(\mathR^2) \cap L^{\infty}(\mathR^2)$. If $f$ is locally $\beta$-H\"older continuous 
	with $\beta \in (1-s,\,1]$, then the minimizer $u$ of the functional $\capF_{K,f}$ is also locally $\beta$-H\"older continuous in $\mathR^2$.
\end{theorem}

We point out that we cannot show the regularity result in any dimension due to the appearance of singularities on the boundary of the levelsets of minimizers. 
However, the two-dimensional case is of particular interest for the application to image denoising.

As discussed in the case of the denoising problem in \cite{CCN01, CCN02, CCN}, our regularity result is based on the following observation: if $u$ is a minimizer of the functional $\capF_{K,f}$, then the super-level set $\{u > t\}$ for each $t \in \mathR$ is also a minimizer of the functional associated with the prescribed nonlocal mean curvature problem
\begin{equation}\label{prescribedNonlocalMCProb}
	\min\left\{ \capE_{K,f,t}(E) \mid \text{$E \subset \mathR^n$: measurable}\right\}
\end{equation}
where we define the functional $\capE_{K,f,t}$ by 
\begin{equation}\nonumber
	\capE_{K,f,t}(E) \coloneqq P_K(E) + \int_{E}(t-f(x))\,dx
\end{equation}
for any measurable set $E \subset \mathR^n$ and $t\in\mathR$. Here $P_K$ is the \textit{nonlocal perimeter} associated with the kernel $K$, which is given as
\begin{equation}\nonumber
	P_K(E) \coloneqq \int_{E}\int_{E^c}K(x-y)\,dx\,dy
\end{equation}
for any measurable set $E \subset \mathR^n$ (see Section \ref{preliminaries} for the detail). If $E_t$ is a minimizer of $\capE_{K,f,t}$ for each $t$ and $\partial E_t$ is smooth ($C^2$-regularity is sufficient), then we can obtain that the boundary $\partial E_t$ satisfies the following \textit{prescribed nonlocal mean curvature equation}
\begin{equation}\nonumber
	H^K_{E_t}(x) + t - f(x) = 0
\end{equation}
for any $x \in \partial E_t$. Here $H^K_{E_t}$ is the so-called \textit{nonlocal mean curvature} defined by
\begin{equation}\label{definitionNonlocalMC}
	H^K_{E_t}(x) \coloneqq \text{p.v.}\int_{\mathR^n}K(x-y)(\chi_{E_t}(x) - \chi_{E_t}(y)) \,dy
\end{equation}
for any $x \in \mathR^n$, where we mean by ``p.v." the Cauchy principal value. Note that, if $K(x) = |x|^{-(N+s)}$, then we denote the nonlocal mean curvature at $x \in \partial E$ associated with $K$ by $H^s_E(x)$. One may observe that, if $\partial E$ is of class $C^{1,\alpha}$ with $\alpha > s$, then $H^s_E$ is finite at each point of $\partial E$.

The idea to show the local H\"older regularity of a minimizer $u$ is based on the observation that the distance between the boundaries of the two super-level sets $\{u>t\}$ and $\{u>t'\}$ for $t,\,t' \in \mathR$ with $t \neq t'$ should not be too close. To observe this, we compare between the nonlocal mean curvatures of $\partial \{u>t\}$ and $\partial \{u>t'\}$ at the points where the smallest distance between the boundaries $\partial \{u>t\}$ and $\partial \{u>t'\}$ is attained. The comparison can be done thanks to the computations of the first variation of the nonlocal mean curvature shown in \cite{DdPW, JuLa}. Thus we are able to derive some local estimate to assert the local H\"older regularity with the assumption of the local H\"older regularity of $f$. 

The organization of this paper is as follows: In Section \ref{preliminaries}, we will introduce the notation
related to the nonlocal total variations. In Section \ref{secELeqforEnergy}, we will show the correspondence between the minimizers of $\capF_{K,f}$ in \eqref{funcDenoisingProb} and the solutions to the nonlocal 1-Laplace equation. In Section \ref{secComparisonMini}, we will give a sort of comparison principle for the minimizers. As a result of this claim, we will show that, if a datum $f$ is bounded, then the minimizer of $\capF_{K,f}$ is also bounded. In Section \ref{secCharacteriMinimizers}, we will show that each super-level set of a minimizer of $\capF_{K,f}$ is also a minimizer of $\capE_{K,f,t}$ for $t\in\mathR$. In Section \ref{secBoundednessSuperLeverlsets}, we will show the boundedness of each super-level set of the minimizer and, moreover, this set can be uniformly bounded whenever the minimizer is bounded from below. Finally, 
by using all the previous results, 
in Section \ref{secLocalHolderConti} we prove the main theorem in this paper on the H\"older regularity of minimizers in two dimensions.

\section{Notation}\label{preliminaries}
In this section, we give several definitions and properties of the space of functions with finite nonlocal total variations. First of all, we define the space $BV_K(\Omega)$ of functions with nonlocal bounded variations associated with the kernel $K$ by
\begin{equation}\label{functionBVwithK}
	BV_K(\Omega) \coloneqq \left\{u \in L^1(\Omega) \mid [u]_{K}(\Omega) <\infty \right\}
\end{equation} 
where we set, for any measurable function $u$,
\begin{equation}\label{seminormNonolocalTV}
	[u]_{K}(\Omega) \coloneqq \frac{1}{2}\int_{\Omega}\int_{\Omega} K(x-y)\,|u(x) - u(y)|\,dx\,dy.
\end{equation}
We observe that the space $BV_K(\Omega)$ coincides with the fractional Sobolev space $W^{s,\,1}(\Omega)$ when the kernel $K$ is given as $K(x)= |x|^{-(n+s)}$ with $s \in (0,\,1)$ (see, for instance, \cite{NPV}).

Secondly, if we set $\Omega = \mathR^n$ and substitute a characteristic function $\chi_E$ of a set $E\subset\mathR^n$ in \eqref{seminormNonolocalTV}, then we obtain the so-called \textit{nonlocal perimeter}. Namely, we define the nonlocal perimeter of a set $E \subset \mathR^n$ associated with the kernel $K$ by
\begin{equation}\label{definitionNonlocalPeri}
	P_K(E) \coloneqq \int_{E}\int_{E^c}K(x-y)\,dx\,dy.
\end{equation}
In the case that $K(x) = |x|^{-(n+s)}$ for $s \in (0,\,1)$, we call $P_K$ the \textit{$s$-fractional perimeter}, and we denote it by $P_s$. 
This notion was introduced by Caffarelli, Roquejoffre, and Savin in \cite{CRS}. The authors' work in \cite{CRS} is motivated by the structure of inter-phases when long-range correlations exist (see also \cite{CaSo, Imbert} for the study of interfaces with fractional mean curvatures). After their work, any problems involving not only the $s$-fractional perimeter but also the nonlocal perimeter with the kernel $K$ were studied by many authors. We leave here a short list of papers, which are related to our problems, for those who are interested in the variational problems involving the nonlocal perimeter \cite{AuKo, AGP, Brezis02, CeNo, DRM, MRT01, MRT02} and the references are therein. 

Next we can consider a localized version of the nonlocal perimeter $P_K$ as follows: let $\Omega \subset \mathR^n$ be any domain. Then the nonlocal perimeter in $\Omega$ associated with the kernel $K$ is given by 
\begin{align}
	P_K(E;\Omega) &\coloneqq 
	\int_{\Omega \cap E}\int_{\Omega \cap E^c} K(x-y)\,dx\,dy + \int_{\Omega \cap E}\int_{\Omega^c \cap E^c} K(x-y) \,dx\,dy \nonumber\\
	&\qquad + \int_{\Omega \cap E^c}\int_{\Omega^c \cap E} K(x-y) \,dx\,dy \nonumber
\end{align} 
for any $E\subset \mathR^n$. Secondly, we give the definition of solutions to the so-called \textit{nonlocal 1-Laplace equations} associated with the kernel $K$.
\begin{definition}\label{defSolutionNonlocal1Lap}
	Let $F : \mathR^n \times \mathR \to \mathR$ be a measurable function in $L^2(\mathR^n \times \mathR)$. We say that $u \in BV_K \cap L^2(\mathR^n)$ is a solution to the nonlocal equation
	\begin{equation}\label{nonlocalSemilinearEq}
		-\Delta^K_1 u(x) = F(x,\,u(x))  \quad \text{for a.e. $x\in \mathR^n$}
	\end{equation}
	if there exists a function $z: \mathR^n \times \mathR^n \to \mathR$ with $|z| \leq 1$ a.e. in $\mathR^n \times \mathR^n$ and $z(x,\,y) = - z(y,\,x)$ for a.e. $(x,\,y) \in \mathR^n \times \mathR^n$ such that
	\begin{equation}\label{ELeqforF}
		\frac{1}{2}\int_{\mathR^n}\int_{\mathR^n}K(x-y)\,z(x,\,y)(v(x) - v(y))\,dx\,dy = \int_{\mathR^n}F(x,\,u(x))\,v(x)\,dx
	\end{equation} 
	for every $v \in C^{\infty}_c(\mathR^n)$ with $[v]_K(\mathR^n) < \infty$ and
	\begin{equation}\nonumber
		z(x,\,y) \in \sgn(u(y) - u(x)) \quad \text{for a.e. $(x,\,y) \in \mathR^n \times \mathR^n$}
	\end{equation}
	where $\sgn(x)$ is a generalized sign function satisfying that 
	\begin{equation}\nonumber
		\sgn(x) \in [-1,\,1], \quad \sgn(x)\,x = |x| \quad \text{for any $x \in \mathR$}.
	\end{equation}
\end{definition}
We mention that the authors in \cite{MRT01, MRT02} give a similar definition of the nonlocal 1-Laplacian associated with an integrable kernel. 

In the present paper, we only consider the case that $F(x,\,u(x)) = u(x) - f(x)$ for a given datum $f$. The concept of the definition is motivated by the Euler-Lagrange equation of the functional
\begin{equation}\nonumber
	\capI_K(u) \coloneqq \frac{1}{2}\int_{\mathR^n} \int_{\mathR^n} K(x-y)|u(x)-u(y)|\,dx\,dy.
\end{equation}
Indeed, when we assume that $u$ is a minimizer of $\capI_K$ and consider the first variation of the functional $\capI_K$, namely, the quantity $\frac{d}{d\varepsilon}\lfloor_{\varepsilon=0}\capI_K(u + \varepsilon\phi)$ for any suitable test function $\phi$, we can formally obtain 
\begin{equation}\nonumber
	\frac{1}{2}\int_{\mathR^n}\int_{\mathR^n}K(x-y)\frac{u(x) - u(y)}{|u(x) - u(y)|}\,(\phi(x) - \phi(y))\,dx\,dy = 0.
\end{equation}
However, it is quite problematic for us to give a rigorous meaning to the ratio $\frac{u(x) - u(y)}{|u(x) - u(y)|}$. To overcome this difficulty, we may apply Definition \ref{defSolutionNonlocal1Lap} and this can be regarded as one of the proper treatments for this issue. Indeed, in Definition \ref{defSolutionNonlocal1Lap}, we may consider the condition that $z(x,\,y)(u(y) - u(x)) = |u(y) - u(x)|$ for a.e. $(x,\,y) \in \mathR^n \times \mathR^n$ with $u(x) \neq u(y)$ as a natural requirement. Note that the framework of solutions in the sense of Definition \ref{defSolutionNonlocal1Lap} has been originally developed by, for instance, Maz\'on, Rossi, and Toledo in \cite{MRT} and can be seen as a nonlocal counterpart of the framework given in \cite{ABCM}, \cite{ACM}, and \cite{MRD}. 

\section{Preliminary results}

\subsection{Euler-Lagrange equation for $\capF_{K,f}$}\label{secELeqforEnergy}
In this section, we show the necessary and sufficient condition for the minimizers of the functional $\capF_{K,f}$ in $\mathR^n$. Before stating the claim, 
we give some conditions on the kernel $K$ which we will assume in the sequel.
\begin{itemize}
	\item[(K1)] $K:\mathR^n \setminus \{0\} \to \mathR$ is a non-negative measurable function.
	\item[(K2)] $K$ is symmetric with respect to the origin, namely $K(-x) = K(x)$ for any $x \in \mathR^n\setminus \{0\}$.
\end{itemize}
We observe that a typical example of the kernel $K$ is given as $K(x)=|x|^{-(n+s)}$ with $s\in(0,\,1)$ and this function satisfies all the above assumptions. 

In the following lemma, we show that the minimizer of $\capF_{K,f}$ satisfies a prescribed nonlocal mean curvature equation. This equation can be regarded as the Euler-Lagrange equation. Moreover, we show that the converse statement is also valid.
\begin{lemma}\label{equivMinimizerAndSolution}
	Assume that the kernel $K$ satisfies (K1) and (K2) and a given datum $f$ is $L^2(\mathR^2)$. If $u \in BV_K \cap L^2(\mathR^n)$ is a minimizer of the functional $\capF_{K,f}$, then $u$ satisfies the equation
	\begin{equation}\label{nonlocalEquation00}
		-\Delta^K_1 u = u -f  \quad \text{in $\mathR^n$}
	\end{equation}
	in the sense of Definition \ref{defSolutionNonlocal1Lap}. Conversely, if $u \in BV_K \cap L^2(\mathR^n)$ is a solution of the equation \eqref{nonlocalEquation00} in the sense of Definition \ref{defSolutionNonlocal1Lap}, then $u$ is a minimizer of $\capF_{K,f}$.
\end{lemma}
\begin{proof}
	First, we recall the definition of the functional $\capI_{K}$ and the non-negativity of $K$ and thus, find that $\capI_K$ is convex, lower semi-continuous, and positive homogeneous of degree one. Then, by using the same argument as in \cite{MRT01, MRT02}, we can show the characterization of the sub-differential of $\capI_{K}(u)$ as follows:
	\begin{align}\label{characteriSubdiffNonlocalfunc}
		&\partial \capI_{K}(u) \nonumber\\
		& \quad = \left\{ v \in L^2(\mathR^n) \mid \text{$-\Delta^K_1 u = v$ in the sense of Definition \ref{defSolutionNonlocal1Lap}} \right\}.
	\end{align}
	Here we recall the definition of the sub-differential $\partial \capE(u)$ for $u \in X$ of the functional $\capE: X \to \mathR\cup\{+\infty\}$ where $X$ is the Hilbert space with the inner product $(\cdot,\cdot)_X$. We say that $v \in X$ belongs to $\partial \capE(u)$ for each $u \in X$ if it holds that, for any $w \in X$,
	\begin{equation}\nonumber
		\capE(w) - \capE(u) \geq (w,\,v)_X.
	\end{equation}
	Note that $u \in X$ is a minimizer of $\capE$ if and only if $0 \in \partial \capE(u)$. Then, from the general theory on the sub-differential, we can also show the identity 
	\begin{equation}\label{characteriSubdifferentialEnergy}
		\partial \capF_{K,f}(u) = \partial \capI_{K}(u) + u-f.
	\end{equation}
	for any $u \in L^2$. Indeed, if $v \in \partial \capF_{K,f}(u)$, then we can compute the functional of $u$ as follows; for any $w \in L^2(\mathR^n)$, 
	\begin{align}\label{characteriSubdiffe01}
		\capI_{K}(w) - \capI_{K}(u) &=  \capF_{K,f}(w) - \capF_{K,f}(u) + \frac{1}{2}\int_{\mathR^n}(u-f)^2\,dx - \frac{1}{2}\int_{\mathR^n}(w-f)^2 \nonumber\\
		&\geq \int_{\mathR^n}v(w-u)\,dx - \frac{1}{2}\int_{\mathR^n}(w-u)(w+u-2f)\,dx \nonumber\\
		&= \int_{\mathR^n}(v-u+f)(w-u)\,dx + \int_{\mathR^n}(u-f)(w-u)\,dx \nonumber\\
		&\qquad - \frac{1}{2}\int_{\mathR^n}(w-u)(w+u-2f)\,dx \nonumber\\
		&= \int_{\mathR^n}(v-u+f)(w-u)\,dx +\frac{1}{2}\int_{\mathR^n}(w-u)^2\,dx \nonumber\\
		&\geq \int_{\mathR^n}(v-u+f)(w-u)\,dx. 
	\end{align}
	Therefore we obtain $v-u+f \in \partial \capI_{K}(u)$. On the other hand, if $v \in \partial \capI_{K}(u)+u-f$, then we can compute in the following manner; for any $w \in L^2(\mathR^n)$, we have
	\begin{align}\label{characteriSubdiffe02}
		\capF_{K,f}(w) - \capF_{K,f}(u) &= \capI_{K}(w) - \capI_{K}(u) + \frac{1}{2}\int_{\mathR^n}(w-f)^2\,dx - \frac{1}{2}\int_{\mathR^n}(u-f)^2   \\
		&\geq \int_{\mathR^n} (v-u+f)(w-u)\,dx + \frac{1}{2}\int_{\mathR^n}(w-u)(w+u - 2f)\,dx \nonumber\\
		&= \int_{\mathR^n}v(w-u)\,dx + \frac{1}{2}\int_{\mathR^n}(w-u)^2\,dx  \nonumber\\
		&\geq \int_{\mathR^n}v(w-u)\,dx, 
	\end{align}
	and thus we have that $v \in \partial \capF_{K,f}(u)$. Therefore, from the computations \eqref{characteriSubdiffe01} and \eqref{characteriSubdiffe02}, we conclude that the first part of the claim is valid. Then, from \eqref{characteriSubdifferentialEnergy}, we can easily obtain the equity
	\begin{align}\label{identityOverall}
		&\partial \capF_{K,f}(u) \nonumber\\
		&\quad = \left\{ v+u-f \in L^2(\mathR^n) \mid \text{$-\Delta^K_1 u = v$ in the sense of Definition \ref{defSolutionNonlocal1Lap}} \right\}.
	\end{align}
	We can readily see that $0 \in \partial \capF_{K,f}(u)$ whenever $u$ is a minimizer of $\capF_{K,f}$ Therefore, we conclude that, if $u$ is a minimizer of $\capF_{K,f}$, then $u$ is a solution of the equation \eqref{nonlocalEquation00}.
	
	Conversely, if $u$ is a solution of the equation \eqref{nonlocalEquation00}, then from \eqref{identityOverall} we have that $0$ belongs to the set in the right-hand side of \eqref{identityOverall}, and thus we obtain $0 \in \partial \capF_{K,f}(u)$.
\end{proof}

\subsection{Comparison between minimizers}\label{secComparisonMini}
In this section, we prove a comparison principle for the minimizers of $\capF_{K,f}$. We assume that $K$ satisfies the assumptions (K1) and (K2) shown in Section \ref{secELeqforEnergy} and the data $f_1$ and $f_2$ satisfy that $f_1 \leq f_2$. Then we show that the minimizers $u_1$ and $u_2$ associated with $f_1$ and $f_2$, respectively, preserves the inequality. Precisely, we prove the following result:
\begin{lemma}\label{comparisonMini}
	Let $f_i$ be in $L^2(\mathR^n)$ for each $i \in \{1,\,2\}$ and $u_i \in BV_K \cap L^2(\mathR^n)$ be a minimizer of $\capF_{K,f_i}$ for each $i\in\{1,\,2\}$. Assume that the kernel $K : \mathR^n \setminus \{0\} \to \mathR$ satisfies (K1) and (K2). If $f_1 \leq f_2$ $\mathcal{L}^n$-a.e. in $\mathR^n$, then $u_1 \leq u_2$ $\mathcal{L}^n$-a.e. in $\mathR^n$.
\end{lemma}
\begin{proof}
	Let $u_1,\,u_2 \in BV_K(\mathR^n)$ be minimizers of $\capF_{K,f}$ associated with given data $f_1,\,f_2 \in L^{2}(\mathR^n)$, respectively. First of all, we prove the following inequality: 
	\begin{equation}\label{submodularNonlocalTotalVari}
		[u_{+}]_{K}(\mathR^n) + [u_{-}]_{K}(\mathR^n) \leq [u_1]_{K}(\mathR^n) + [u_2]_{K}(\mathR^n).
	\end{equation}
	Indeed, setting  
	\begin{equation}\label{maxMin}
		u_{+}(x) \coloneqq \max\{u_1(x),\,u_2(x)\}, \quad u_{-}(x) \coloneqq \min\{u_1(x),\,u_2(x)\}
	\end{equation}
	for any $x \in \mathR^n$ and by the co-area formula, we have that
	\begin{align}\label{coareaFormulaTV}
		[u_{i}]_{K}(\mathR^n) &= \int_{-\infty}^{\infty}\,\frac{1}{2}\int_{\mathR^n}\int_{\mathR^n}K(x-y)\,|\chi_{\{u_{i}>t\}}(x) - \chi_{\{u_{i}>t\}}(y)| \,dx\,dy\,dt \nonumber\\
		&= \int_{-\infty}^{\infty} P_K(\{u_{i} > t\}) \,dt 
	\end{align}
	for any $i \in \{1,\,2,\,+,\,-\}$. We recall that the nonlocal perimeter $P_K$ is sub-modular, namely, it holds that
	\begin{equation}\label{submodularNonlocalPeri}
		P_K(E \cup F) + P_K(E \cap F) \leq P_K(E) + P_K(F)
	\end{equation}
	for any $E,\,F\subset \mathR^n$. Therefore from \eqref{submodularNonlocalPeri} and the definitions of $u_{+}$ and $u_{-}$, we obtain the claim.
	
	Now from the general theory of calculus of variations, the minimizer of $\capF_{K,f}$ is unique in $L^2(\mathR^n)$ and thus, it is sufficient to prove that 
	\begin{equation}\nonumber
		\capF_{K,f_2}(u_{+}) \leq \capF_{K,f_2}(u_2)
	\end{equation} 
	where $u_{+}$ is defined in \eqref{maxMin} to obtain the lemma. From a simple computation, we can easily see that the inequality
	\begin{equation}\label{computationMinimizers}
		(u_{-}- f_1)^2 + (u_{+}- f_2)^2 \leq (u_1- f_1)^2 + (u_2- f_2)^2
	\end{equation}
	in $\mathR^n$. From the minimality of $u_i$ for $i\in\{1,\,2\}$, we have
	\begin{equation}\label{estimate01}
		\capF_{K,f_1}(u_1) + \capF_{K,f_2}(u_2) \leq \capF_{K,f_1}(u_{-}) + \capF_{K,f_2}(u_{+}). 
	\end{equation}
	On the other hand, from \eqref{submodularNonlocalTotalVari} and \eqref{computationMinimizers}, we have
	\begin{align}\label{estimate02}
		&\capF_{K,f_1}(u_{-}) + \capF_{K,f_2}(u_{+})  \\
		&\leq [u_{-}]_{K}(\mathR^n) + \frac{1}{2}\int_{\mathR^n}(u_{-}- f_1)^2\,dx + [u_{+}]_{K}(\mathR^n) + \frac{1}{2}\int_{\mathR^n}(u_{+}- f_2)^2\,dx \nonumber\\
		&= [u_1]_{K}(\mathR^n) + \frac{1}{2}\int_{\mathR^n}(u_1- f_1)^2\,dx + [u_2]_{K}(\mathR^n) + \frac{1}{2}\int_{\mathR^n}(u_2- f_2)^2\,dx \nonumber\\
		&\quad +\frac{1}{2}\int_{\mathR^n}(u_{-}- f_1)^2\,dx - \frac{1}{2}\int_{\mathR^n}(u_1- f_1)^2\,dx \nonumber\\
		&\quad \quad + \frac{1}{2}\int_{\mathR^n}(u_{+}- f_2)^2\,dx - \frac{1}{2}\int_{\mathR^n}(u_2- f_2)^2\,dx \nonumber\\
		&\leq \capF_{K,f_1}(u_1) + \capF_{K,f_2}(u_2). 
	\end{align}
	Thus from \eqref{estimate01} and \eqref{estimate02}, we obtain
	\begin{equation}\label{identityEnergies}
		\capF_{K,f_1}(u_1) + \capF_{K,f_2}(u_2) = \capF_{K,f_1}(u_{-}) + \capF_{K,f_2}(u_{+})
	\end{equation}
	Now suppose by contradiction that $\capF_{K,f_2}(u_{+}) > \capF_{K,f_2}(u_2)$. Then from \eqref{identityEnergies} we have
	\begin{equation}\nonumber
		\capF_{K,f_1}(u_1) > \capF_{K,f_1}(u_{-})
	\end{equation}
	which contradicts the minimality of $u_1$. Thus we obtain the inequality $\capF_{K,f_2}(u_{+}) \leq \capF_{K,f_2}(u_2)$. Therefore, by the uniqueness of the minimizer of $\capF_{K}$ in $L^2(\mathR^n)$, we obtain that $u_{+} = u_2$ a.e. in $\mathR^n$, which implies that $u_2 \geq u_1$ a.e. in $\mathR^n$.
\end{proof}

\begin{corollary}\label{nonnegativityMini}
	Assume that the kernel $K: \mathR^n \setminus \{0\} \to \mathR$ satisfies the assumptions (K1) and (K2) in Section \ref{secELeqforEnergy}. If a datum $f \in L^2(\mathR^n)$ is non-negative a.e. in $\mathR^n$, then the minimizer $u \in BV_K \cap L^2(\mathR^n)$ is also non-negative a.e. in $\mathR^n$. 
\end{corollary}
\begin{proof}
	Since it holds that 
	\begin{equation}\nonumber
		\capF_{K,0}(0) = 0 \leq \capF_{K,0}(v)
	\end{equation}
	for every $v \in BV_K \cap L^2(\mathR^n)$, we have that the unique solution of the problem
	\begin{equation}\nonumber
		\inf\{\capF_{K,0}(v) \mid v \in BV_K \cap L^2\}
	\end{equation}
	is $v=0$. Hence, by applying Lemma \ref{comparisonMini} to the case that $f_1=0$ and $f_2=f$, we obtain that $0 \leq u$ a.e. in $\mathR^n$.
\end{proof}

Finally, we show a sort of comparison property of minimizers under the assumption that a datum $f$ is bounded in $\mathR^n$. We do not derive the following proposition directly from Lemma \ref{comparisonMini} but from a simple computation. 
\begin{proposition}\label{comparisonLinfty}
	Let $u \in BV_K \cap L^2(\mathR^n)$ be a minimizer of $\capF_{K,f}$ with a datum $f \in L^2(\mathR^n)$. Assume that the kernel $K:\mathR^n \setminus \{0\} \to \mathR$ is non-negative measurable function. If there exists a constant $C>0$ such that $|f(x)| \leq C$ for a.e. $x \in \mathR^n$, then $|u(x)| \leq C$ for a.e. $x \in \mathR^n$ with the same constant $C$.
\end{proposition} 
\begin{proof}
	It is sufficient to show that, if $f \leq  C$ a.e. in $\mathR^n$ with some constant $C>0$, then $u \leq  C$ a.e. in $\mathR^n$ with the same constant $C$ because we only repeat the same argument as we show in this proof. We define $v(x) \coloneqq \min\{u(x),\,C\}$ for $x \in \mathR^n$. It is sufficient to show that $u = v$ for a.e. in $\mathR^n$. From the definition, we can show the claim that $|v(x) - v(y)| \leq |u(x) - u(y)|$ for $x,\,y \in \mathR^n$. Indeed, if $u(x) \leq C$ and $u(y) \leq C$ or $u(x) > C$ and $u(y) > C$, then we can readily obtain the claim. If $u(x) \leq C$ and $u(y) > C$, then we have
	\begin{align}
		|u(x) - u(y)|^2 - |v(x) - v(y)|^2 &= u^2(y) - C^2 - 2u(x)\,u(y) + 2u(x)\,C \nonumber\\
		&= (u(y)-C)(u(y) + C - 2u(x)) \geq 0. \nonumber
	\end{align}
	In the same way, we can prove the claim if $u(x) > C$ and $u(y) \leq C$. Moreover, we can show that $(v-f)^2 \leq (u-f)^2$ in $\mathR^n$.	Therefore we compute the functional associated with $v$ as follows:
	\begin{align}
		\capF_{K,f}(v) &= \frac{1}{2}\int_{\mathR^n}\int_{\mathR^n}K(x-y)\,|v(x)-v(y)|\,dx\,dy + \frac{1}{2}\int_{\mathR^n}(v-f)^2 \,dx \nonumber\\
		& \leq \frac{1}{2}\int_{\mathR^n}\int_{\mathR^n}K(x-y)\,|u(x)-u(y)|\,dx\,dy + \frac{1}{2}\int_{\mathR^n}(u-f)^2 \,dx \nonumber\\
		&= \capF_{K,f}(u). \nonumber
	\end{align}
	Thus, from the uniqueness of the minimizer of $\capF_{K,f}$ in $L^2(\mathR^n)$, we obtain $v = u$ a.e. in $\mathR^n$ and this concludes the proof.
\end{proof}

\subsection{Characterization of minimizers for $\capF_{K,f}$}\label{secCharacteriMinimizers}
In this section, we show the following claim which gives a relation between the minimizers of $\capF_{K,f}$ and $\capE_{K,f,t}$ for $t \in \mathR$. Recall that $\capE_{K,f,t}(E)$ as
\begin{equation}\label{nonlocalPeriEnergy}
	\capE_{K,f,t}(E) \coloneqq P_K(E) + \int_{E}(t-f(x))\,dx
\end{equation}
for every measurable set $E \subset \mathR^n$ where we assume that $f \in L^2(\mathR^n)$ is a given datum and $t\in\mathR$ is any number. 

\begin{lemma}\label{relationMiniTwoEnergies}
	Assume that the kernel $K(x) = |x|^{-(n+s)}$ for $x \in \mathR^n \setminus \{0\}$ with $s \in (0,\,1)$ and a datum $f \in L^2 \cap L^{\infty}(\mathR^n)$. If $u \in BV_K \cap L^2(\mathR^n)$ be a minimizer of $\capF_{K,f}$, then the set $\{x\in\mathR^n \mid u(x)>t\}$ is also a minimizer of $\capE_{K,f,t}(E)$ for every $t\in\mathR$ among measurable sets $E \subset \mathR^n$.
\end{lemma}
\begin{proof}
	Let $F \subset \mathR^n$ be any measurable set. We may assume that $P_K(F) < \infty$; otherwise this set cannot minimize the functional $\capE_{K,f,t}$. Moreover, we may assume that $\|\chi_F\|_{L^1} = |F| <\infty$ because of the nonlocal isoperimeteric inequality. Then it suffices to show that the super-level set $\{u > t\}$ for each $t\in\mathR$ satisfies the inequality
	\begin{equation}\label{minimalityInequality}
		P_K(\{u > t\}) + \int_{\{u > t\}}(t-f(x))\,dx \leq P_K(F) + \int_{F}(t-f(x))\,dx.
	\end{equation}
	From Lemma \ref{equivMinimizerAndSolution} and the assumption that $u$ is a minimizer of the functional $\capF_{K,f}$, we have that $u$ is also a solution of the equation
	\begin{equation}\label{nonlocalEq}
		-\Delta^K_1 u = u -f  \quad \text{ in }\mathR^n.
	\end{equation}
	Thus, from Definition \ref{defSolutionNonlocal1Lap}, there exists a function $z_u\in L^{\infty}(\mathR^n \times \mathR^n)$ with $|z_u| \leq 1$ and $z_u$ being antisymmetric such that
	\begin{equation}\label{defSolNonlocalEq}
		\frac{1}{2}\int_{\mathR^n}\int_{\mathR^n}K(x-y)\,z_u(x,\,y)\,(w(x) - w(y))\,dx\,dy = \int_{\mathR^n}(u-f)\,w(x)\,dx
	\end{equation}
	for any $w \in C^{\infty}_c(\mathR^n)$ with $[w]_K(\mathR^n) < \infty$ and moreover 
	\begin{equation}\label{propSolNonlocalEq}
		z_u(x,\,y)(u(y) - u(x)) = |u(y) - u(x)|
	\end{equation}
	for a.e. $x,\,y\in \mathR^n$. From the co-area formula, we have the following two identities:
	\begin{equation}\label{coarea01}
		|u(x) - u(y)| = \int_{-\infty}^{+\infty}|\chi_{\{u > t\}}(x) - \chi_{\{u > t\}}(y)|\,dt
	\end{equation}
	and
	\begin{equation}\label{coarea02}
		(u(x) - u(y)) = \int_{-\infty}^{+\infty}(\chi_{\{u > t\}}(x) - \chi_{\{u > t\}}(y))\,dt
	\end{equation}
	for any measurable $u:\mathR^n \to \mathR$ and a.e. $x,\,y \in \mathR^n$. Thus from \eqref{propSolNonlocalEq}, \eqref{coarea01}, and \eqref{coarea02}, we obtain 
	\begin{equation}\label{identityLevelset}
		z_u(x,\,y)(\chi_{\{u > t\}}(y) - \chi_{\{u > t\}}(x)) = |\chi_{\{u > t\}}(y) - \chi_{\{u > t\}}(x)|
	\end{equation}
	for a.e. $t \in \mathR$. Now we fix $t \in \mathR$ such that \eqref{identityLevelset} holds. From the specific choice of $K(x) = |x|^{-(n+s)}$, the function space $BV_K(\mathR^n)$ coincides with the fractional Sobolev space $W^{s,1}(\mathR^n)$. Recall that the space $C^{\infty}_c(\mathR^n)$ of smooth functions with compact supports is dense in $W^{s,1}(\mathR^n)$ (see \cite{Adam} for the detail). Hence, from the fact that $P_K(\{u>t\})$ and $P_K(F)$ are finite, we can choose sequences $\{\eta_{l}^u\}_{l\in\mathN}$ and $\{\eta_{l}^F\}_{l\in\mathN}$ in $C^{\infty}_c(\mathR^n)$ such that
	\begin{equation}\label{approximationCharacterFunc}
		\eta_{l}^u \xrightarrow[l \to \infty]{} \chi_{\{u>t\}}, \quad \eta_{l}^F \xrightarrow[l \to \infty]{} \chi_{F} \quad \text{in $W^{s,1}(\mathR^n)$}.
	\end{equation}
	From the choice of the approximation, we notice that the difference function $\eta_{l}^u - \eta_{l}^F$ is also in $C^{\infty}_c(\mathR^n)$ and $[\eta_{l}^u - \eta_{l}^F]_K(\mathR^n) < \infty$ for each $l \in \mathN$. Hence, from the definition of solutions to the equation \eqref{nonlocalEq},  we obtain
	\begin{align}\label{identitySubstitutionCompetitor}
		&\int_{\mathR^n}(u-f)\,(\eta_{l}^u - \eta_{l}^F)\,dx \nonumber\\
		&= -\frac{1}{2} \int_{\mathR^n}\int_{\mathR^n}K(x-y)\,z_u(x,\,y)\,[(\eta_{l}^u - \eta_{l}^F)(y) - (\eta_{l}^u - \eta_{l}^F)(x)]\,dx\,dy \nonumber\\
		&= -\frac{1}{2} \int_{\mathR^n}\int_{\mathR^n}K(x-y)\,z_u(x,\,y)\,(\eta_{l}^u(y) - \eta_{l}^u(x))\,dx\,dy \nonumber\\
		&\qquad +  \frac{1}{2} \int_{\mathR^n}\int_{\mathR^n}K(x-y)\,z_u(x,\,y)\,(\eta_{l}^F(y) - \eta_{l}^F(x))\,dx\,dy. 
	\end{align}
	By applying Proposition \ref{comparisonLinfty} and from the assumption that $f \in L^{\infty}(\mathR^n)$, we have that the minimizer $u$ is also in $L^{\infty}(\mathR^n)$ and thus
	\begin{equation}\label{convergenceLinearTerm}
		\left|\int_{\mathR^n}(u-f)(\eta_{l}^u - \eta_{l}^F)\,dx - \int_{\mathR^n}(u-f)(\chi_{\{u>t\}} - \chi_{F})\,dx \right| \xrightarrow[l \to \infty]{} 0.
	\end{equation}
	Hence by applying the dominated convergence theorem and from \eqref{approximationCharacterFunc}, \eqref{identitySubstitutionCompetitor}, and \eqref{convergenceLinearTerm}, we obtain that
	\begin{align}\label{identitySubstitutionLimit}
		&\int_{\mathR^n}(u-f)(\chi_{\{u>t\}} - \chi_{F})\,dx \nonumber\\
		&= \lim_{l \to \infty}\int_{\mathR^n}(u-f)\,(\eta_{l}^u - \eta_{l}^F) \,dx \nonumber\\
		&= -\frac{1}{2} \int_{\mathR^n}\int_{\mathR^n}K(x-y)\,z_u(x,\,y)\,(\chi_{\{u>t\}}(y) - \chi_{\{u>t\}}(x))\,dx\,dy \nonumber\\
		&\qquad -\frac{1}{2} \int_{\mathR^n}\int_{\mathR^n}K(x-y) \,z_u(x,\,y)\,(\chi_{F}(y) - \chi_{F}(x))\,dx\,dy. 
	\end{align} 
	From the definition of $z_u$, we have
	\begin{align}\label{perimeterAnySetF}
		&\frac{1}{2} \int_{\mathR^n}\int_{\mathR^n}K(x-y)\,z_u(x,\,y)\,(\chi_{F}(x) - \chi_{F}(y))\,dx\,dy \nonumber\\
		&\leq \frac{1}{2} \int_{\mathR^n}\int_{\mathR^n}K(x-y)|\chi_{F}(x) - \chi_{F}(y)|\,dx\,dy = P_K(F). 
	\end{align}
	Taking into account \eqref{identityLevelset}, \eqref{identitySubstitutionLimit}, and \eqref{perimeterAnySetF}, we obtain
	\begin{align}\label{periEstimateOverall01}
		&\int_{\mathR^n}(u-f)\,(\chi_{\{u>t\}} - \chi_{F})\,dx \nonumber\\
		&\leq - \frac{1}{2} \int_{\mathR^n}\int_{\mathR^n}K(x-y)|\chi_{\{u>t\}}(x) - \chi_{\{u>t\}}(y)| \,dx\,dy + P_{K}(F). 
	\end{align}
	Regarding the left-hand side of \eqref{periEstimateOverall01}, we have
	\begin{align}\label{estimateLHS}
		\int_{\mathR^n}(u-f)\,(\chi_{\{u>t\}} - \chi_{F}) \,dx &= \int_{\mathR^n}(u-t+t-f)\,(\chi_{\{u>t\}} - \chi_{F}) \,dx \nonumber\\
		&\geq \int_{\{u > t\} \cap F^c} (t-f) \,dx  - \int_{\{u \leq t\} \cap F} (u-f) \,dx  \nonumber\\
		&\geq  \int_{\{u > t\} \cap F^c} (t-f) \,dx  - \int_{\{u \leq t\} \cap F} (t-f) \,dx  \nonumber\\
		&= \int_{\mathR^n}(t-f)\,(\chi_{\{u>t\}} - \chi_{F})\,dx 
	\end{align}
	for a.e. $t\in\mathR$. Hence, from \eqref{periEstimateOverall01} and \eqref{estimateLHS}, we have
	\begin{align}\label{periEstimateOverall02}
		&P_K(\{u>t\}) + \int_{\mathR^n}(t-f)\,\chi_{\{u > t\}}\,dx \nonumber\\
		&= \frac{1}{2} \int_{\mathR^n}\int_{\mathR^n}K(x-y)|\chi_{\{u>t\}}(x) - \chi_{\{u>t\}}(y)|\,dx\,dy + \int_{\mathR^n}(t-f)\,\chi_{\{u > t\}}\,dx  \nonumber\\
		&\leq P_{K}(F) +  \int_{\mathR^n} (t-f)\,\chi_{F}\,dx 
	\end{align}
	for a.e. $t\in\mathR$. Therefore we conclude that the inequality \eqref{minimalityInequality} holds for a.e. $t\in\mathR$. Notice that, for any $t\in\mathR$ such that \eqref{identityLevelset} does not hold, we can choose a sequence $\{t_j\}_{j\in\mathN}$ such that $t_j \rightarrow t$ as $j \to \infty$ and \eqref{identityLevelset} holds for any $t_j$; otherwise we can choose a constant $\delta>0$ such that $B_{\delta}(t) \subset \{t\in\mathR \mid \text{\eqref{identityLevelset} is not true}\}$. Since the condition \eqref{identityLevelset} holds true for a.e. $t \in \mathR$, we have that
	\begin{equation}\nonumber
		0< 2\delta = |B_{\delta}(t)| \leq |\{t\in\mathR \mid \text{\eqref{identityLevelset} is not true}\}| = 0,
	\end{equation}
	which is a contradiction. Thus from the lower semi-continuity of $P_K$ and the continuity of the map $t \mapsto |\{u>t\}|$, we conclude that \eqref{minimalityInequality} holds for every $t\in\mathR$. 
\end{proof}

\subsection{Boundedness of super-level sets of minimizers}\label{secBoundednessSuperLeverlsets}
Let $u \in BV_K \cap L^2(\mathR^n)$ be a minimizer of $\capF_{K,f}$ with a datum $f \in L^p(\mathR^n)$ with $p \in (\frac{n}{s},\,\infty]$. In this section, we show that the super-level set $\{u>t\}$ for each $t \in \mathR$ is bounded up to negligible sets. Precisely, we prove  
\begin{lemma}\label{boundednessMinimizers}
	Assume that the kernel $K(x) = |x|^{-(n+s)}$ for $x \in \mathR^n \setminus \{0\}$ with $s \in (0,\,1)$ and $f \in L^p(\mathR^n)$ with $p \in (\frac{n}{s},\,\infty]$. If $E_T$ is a minimizer of $\capE_{K,f,T}$ among sets with finite volumes for any $T\in\mathR$, then there exists a constant $R_T >0$ such that $|E_T \setminus B_{R_T}|=0$.
\end{lemma}
\begin{proof}
	We basically follow the proof shown in \cite[Proposition 3.2]{CeNo}. Suppose by contradiction that $|E_T \setminus B_r| > 0$ for any $r > 0$. By setting $\phi_T(r) \coloneqq |E_T \setminus B_r|$ for any $r >0$, we have 
	\begin{equation}\nonumber
		(\phi_T)'(r) = - \capH^{n-1}(E_T \cap \partial B_r)
	\end{equation}
	for a.e. $r>0$. We fix any $R > 1$. From the minimality of $E_T$, we have
	\begin{equation}\label{inequalityByMinimality00}
		\capE_{K,f,T}(E_T) \leq \capE_{K,f,T}(E_T \cap B_r).
	\end{equation}
	Since it holds that 
	\begin{equation}\nonumber
		P_K(A \cup B) = P_K(A) + P_K(B) - 2 \int_{A}\int_{B} K(x-y)\,dx\,dy
	\end{equation}
	for sets $A,\,B \subset \mathR^n$ with $A\cap B = \emptyset$, we have
	\begin{equation}\label{inequalityByMinimality}
		P_K(E_T\setminus B_r) \leq 2\int_{ E_T \cap B_r}\int_{E_T\setminus B_r
		} K(x-y)\,dx\,dy - \int_{E_T \setminus B_r}(T - f(x))\,dx.
	\end{equation}
	From the isoperimetric inequality for the nonlocal perimeter, we can have the following lower bound of the term of the left-hand side in \eqref{inequalityByMinimality} (see for instance \cite{FFMMM}):
	\begin{equation}\label{isoperiNonlocalPeri}
		P_K(E_T\setminus B_r) \geq \frac{P_K(B_1)}{|B_1|^{\frac{n-s}{n}}}\,|E_T\setminus B_r|^{\frac{n-s}{n}} = C(n,s)\,\phi_T^{\frac{n-s}{n}}(r)
	\end{equation}
	for $r \geq R$, where we set $C(n,s) \coloneqq |B_1|^{-\frac{n-s}{n}}\,P_K(B_1)$. Secondly, from Fubini-Tonelli's theorem and the co-area formula, we can compute the first term of the right-hand side in \eqref{inequalityByMinimality} as follows:
	\begin{align}\label{estifromMini01}
		\int_{E_T \cap B_r}\int_{E_T\setminus B_r} K(x-y)\,dx\,dy &\leq \int_{E_T\setminus B_r}\int_{B_{|y|-r}(y)} \frac{1}{|x-y|^{n+s}}\,dx\,dy \nonumber\\
		&=  \int_{E_T\setminus B_r}|\mathS^{n-1}|\int_{|y|-r}^{\infty} \frac{1}{r^{1+s}}\,dr\,dy \nonumber\\
		&\leq \frac{|\mathS^{n-1}|}{s}\int_{E_T \setminus B_r}(|y|-r)^{-s}\,dy \nonumber\\
		&= \frac{|\mathS^{n-1}|}{s} \int_{r}^{+\infty}\frac{\capH^{n-1}(E_T \cap \partial B_{\sigma})}{(\sigma - r)^s}\,d\sigma \nonumber\\
		&= -\frac{|\mathS^{n-1}|}{s} \int_{r}^{+\infty} \frac{(\phi_T)'(\sigma)}{(\sigma - r)^s}\,d\sigma 
	\end{align}
	for any $r \geq R$. Finally, regarding the second term of the right-hand side in \eqref{inequalityByMinimality}, from the assumption of $f$ and Cauchy-Schwartz inequality (if $p \neq \infty$), we have 
	\begin{align}\label{estifromMini02}
		\int_{E_T \setminus B_r }(-T + f(x))\,dx &\leq T\,|E_T \setminus B_r| + \|f\|_{L^p(\mathR^n)}\,|E_T \setminus B_r|^{\frac{1}{q}} \nonumber\\
		&= T\,\phi_T(r) +  \|f\|_{L^p(\mathR^n)}\,\phi_T^{\frac{1}{q}}(r) < \infty 
	\end{align}
	for any $r \geq R>1$ where $q \geq 1$ satisfies $p^{-1}+q^{-1}=1$. By combining all the computations \eqref{isoperiNonlocalPeri}, \eqref{estifromMini01}, and \eqref{estifromMini02} with \eqref{inequalityByMinimality}, we obtain
	\begin{equation}\label{estiByMinimality}
		C(n,s)\,\phi_T^{\frac{n-s}{n}}(r) \leq -C_1 \int_{r}^{+\infty} \frac{(\phi_T)'(\sigma)}{(\sigma - r)^s}\,d\sigma + T\,\phi_T(r) +  \|f\|_{L^p(\mathR^n)}\,\phi_T^{\frac{1}{q}}(r)
	\end{equation}
	for any $r \geq R$ where we set $C_1 \coloneqq \frac{|\mathS^{n-1}|}{s}$. Since $\phi_T(r)$ vanishes as $r \to \infty$ and $\frac{1}{q} > \frac{n-s}{n}$, we can have that
	\begin{equation}\nonumber
		2T\,\phi_T(r) + 2\|f\|_{L^p(\mathR^n)}\phi_T^{\frac{1}{q}}(r) \leq C(n,s)\,\phi_T^{\frac{n-s}{n}}(r)
	\end{equation}
	for sufficiently large $r \geq R$. Hence, by integrating the both sides of \eqref{estiByMinimality} over $r \in (R,\infty)$, we obtain
	\begin{equation}\label{estiByMinimality01}
		\frac{C(n,s)}{2}\int_{R}^{\infty}\phi_T^{\frac{n-s}{n}}(r)\,dr \leq -C_1 \int_{R}^{\infty}\int_{r}^{+\infty} \frac{(\phi_T)'(\sigma)}{(\sigma - r)^s}\,d\sigma\,dr.
	\end{equation}
	By exchanging the order of the integration, we have
	\begin{equation}\label{exchangeIntergration}
		\int_{R}^{\infty}\int_{r}^{+\infty} \frac{(\phi_T)'(\sigma)}{(\sigma - r)^s}\,d\sigma\,dr = \int_{R}^{\infty}\int_{R}^{\sigma} \frac{(\phi_T)'(\sigma)}{(\sigma - r)^s}\,dr\,d\sigma.
	\end{equation}
	Then by employing the similar computation shown in \cite{CeNo}, we obtain
	\begin{equation}\nonumber
		\int_{R}^{\infty}\int_{R}^{\sigma} \frac{(\phi_T)'(\sigma)}{(\sigma - r)^s}\,dr\,d\sigma \geq -\frac{\phi_T(R)}{1-s} - \int_{R+1}^{\infty} \frac{\phi_T(r)}{(\sigma - R)^s}\,d\sigma. 
	\end{equation}
	Therefore, from \eqref{estiByMinimality01}, we have
	\begin{align}
		\frac{C(n,s)}{2}\int_{R}^{\infty}\phi_T^{\frac{n-s}{n}}(r)\,dr &\leq C_1\frac{\phi_T(R)}{1-s} + C_1\int_{R+1}^{\infty} \frac{\phi_T(\sigma)}{(\sigma - R)^s}\,d\sigma \nonumber\\
		&\leq C_1\frac{\phi_T(R)}{1-s} + C_1\int_{R+1}^{\infty}\phi_T(\sigma)\,d\sigma. \nonumber
	\end{align}
	Again, by choosing $R$ sufficiently large so that the inequality 
	\begin{equation}\nonumber
		C_1\int_{R+1}^{\infty}\phi_T(r)\,dr \leq \frac{C(n,s)}{4}\int_{R}^{\infty}\phi_T^{\frac{n-s}{n}}(r)\,dr
	\end{equation}
	holds, we have
	\begin{equation}\nonumber
		\int_{R}^{\infty}\phi_T^{\frac{n-s}{n}}(r)\,dr \leq \frac{4C_1}{C(n,s)(1-s)}\,\phi_T(R).
	\end{equation}
	Then by applying the method shown in, for instance, \cite{CNRV, CeNo}, we obtain the contradiction to the assumption that $\phi_T(r)>0$ for any $r>0$. Therefore, we conclude the essential boundedness of the set $E_T$.
\end{proof}
We assume that $u \in BV_K \cap L^2(\mathR^n)$ is a minimizer of the functional $\capF_{K,f}$ and $u$ is bounded from below with the constant $c \in \mathR$. Then, since the super-level set $\{u > c\}$ is also a minimizer of $\capE_{K,f,c}$, we may obtain from Lemma \ref{boundednessMinimizers} that there exists a constant $R_c>1$ such that $|\{u>c\} \setminus B_{R_c}| = 0$. In addition to this, we have the inclusion of the super-level sets that $\{u>t'\} \subset \{u>t\}$ for any $t' > t$. Thus, we conclude that the following corollary holds.
\begin{corollary}\label{corUniformBoundednessMinimizer}
	Assume that the kernel $K(x) = |x|^{-(n+s)}$ for $x \in \mathR^n \setminus \{0\}$ with $s \in (0,\,1)$. Let $u \in BV_K \cap L^2(\mathR^n)$ be a minimizer of $\capF_{K,f}$. If a datum $f$ is in $L^p(\mathR^n)$ with $p \in (\frac{n}{s},\,\infty]$ and $u \geq c$ a.e. in $\mathR^n$ for some $c \in \mathR$, then the super-level set $\{u>t\}$ is uniformly bounded with respect to $t \geq c$. Namely, there exists $R_c>0$, independent of $t$, such that $\{u>t\} \subset B_{R_c}$ for any $t \geq c$.
\end{corollary}

\section{H\"older regularity of minimizers}\label{secLocalHolderConti}

First of all, we observe that, if $u$ satisfies some equation associated with the Euler-Lagrange equations of $\capE_{K,f,t}$ and the boundary of $\{u>t\}$ is regular, then $u$ is continuous. 

\begin{proposition}\label{continuityMinimizersLemma}
	Assume that $K(x) = |x|^{-(N+s)}$ for any $x \in \mathR^n$ with $s \in (0,\,1)$ and the datum $f$ is in $L^2 \cap L^{\infty}(\mathR^n)$. 
	Let $u \in BV_K \cap L^2(\mathR^n)$. Assume that $\partial \{u > t\}$ is of class $C^{1,\alpha}$ with $\alpha \in (s,\,1]$ and $u$ satisfies the equation
	\begin{equation*}
		H^s_{\{u>t\}}(x) + t - f(x) = 0
	\end{equation*}
	for any $x \in \partial \{u>t\}$ and $t\in\mathR$. Then $u$ is continuous in $\mathR^n$.
\end{proposition}
\begin{proof}
	Suppose by contradiction that $u$ is not continuous in $\mathR^n$. Then there exist a point $x_0 \in \mathR^n$ and $-\infty< t' < t < \infty$ such that $x_0 \in \partial E_t \cap \partial E_{t'}$. Indeed, if $u$ is not continuous at $x_0$, then it holds that $t_{+} \coloneqq \limsup_{x \to x_0}u(x) > \liminf_{x \to x_0}u(x) \eqqcolon t_{-}$. Note that $t_{+} \geq u(x_0) \geq t_{-}$ by definition. 
	Setting $\delta \coloneqq t_{+} - t_{-}>0$ and the definition of $t_{+}$, we can choose a sequence $\{x_n\}_{n\in\mathN}$ such that $x_k \to x_0$ and $u(x_k) > t_{+} - \frac{\delta}{2^{k}}$ for any $k\in\mathN$ with $k \geq 1$. If $u(x_0) = t_{+}$, then we have that $x_k \in \{u > u(x_0) - \frac{\delta}{2}\}$ for large $k\in\mathN$. Thus we obtain that $x_0 \in \overline{\{u > u(x_0) - \frac{\delta}{2} \}}$. However, from the definition of $\delta$, $x_0$ cannot be a interior point of $\{u > u(x_0) - \frac{\delta}{2} \}$; otherwise we can choose a sequence $\{y_k\}_{k\in\mathN}$ such that 
	\begin{equation}
		u(x_0) - \frac{\delta}{2} < u(y_k) < t_{-} + \frac{\delta}{2^{k}}
	\end{equation}
	for any large $k$. From the definition of $\delta$ and the fact that $u(x_0)=t_{+}$, we obtain a contradiction. Thus we may assume that $u(x_0) < t_{+}$. Setting $\tilde{\delta} \coloneqq t_{+} - u(x_0)>0$ and since $u(x_k) > t_{+} - \frac{\delta}{2^{k}}$ for any $k\in\mathN$, we have that $u(x_k) > u(x_0) + \frac{1}{2}\tilde{\delta}$ for any $k\in\mathN$ with $k \geq (2\delta)^{-1}\tilde{\delta}$ and that $x_k \in \{u > u(x_0) + \frac{1}{2}\tilde{\delta}\}$ for large $k\in\mathN$. Hence, recalling that $x_k \to x_0$ as $k\to\infty$, we obtain that $x_0 \in \partial \{u > u(x_0) + \frac{1}{2}\tilde{\delta} \}$. In the same way, we can show that $x_0 \in \partial \{u > u(x_0) + \frac{3}{4}\tilde{\delta} \}$. Therefore, we conclude that, if $u$ is not continuous at $x_0$, we can find distinct constants $t,\,t'\in\mathR$ such that $x_0 \in \partial \{u>t\} \cap \partial \{u>t'\}$.
	
	From the assumptions, we obtain that the following equations hold:
	\begin{equation}\label{eulerLagrange01}
		H_{E_t}^s(x) + t - f(x)  = 0
	\end{equation}
	and
	\begin{equation}\label{eulerLagrange02}
		H_{E_{t'}}^s(x) + t' - f(x) = 0
	\end{equation} 
	for each $x \in \partial E_t \cap \partial E_{t'}$. Recall that the nonlocal mean curvature associated with $K(x) = |x|^{-(N+s)}$ is well-defined at each point on $\partial E$ if $\partial E$ is at least of class $C^{1,\alpha}$ with $\alpha>s$ (see, for instance, \cite[Corollary 3.5]{Cozzi}).
	
	Now we can readily see that, if two sets $E$ and $F$ satisfy that $E \subset F$ and $\partial E \cap \partial F \not= \emptyset$, then it holds that $H^s_{E} \geq H^s_{F}$ on $\partial E \cap \partial F$. Indeed, by definition, we have
	\begin{align}\label{nonlocalMeanCurvature}
		H^s_{E}(x) - H^s_{F}(x) &= \text{P.V.}\,\int_{\mathR^n}\frac{\chi_{E}(x) - \chi_{E}(y)}{|x-y|^{N+s}}\,dy \nonumber\\
		&\qquad - \text{p.v.}\,\int_{\mathR^n}\frac{\chi_{F}(x) - \chi_{F}(y)}{|x-y|^{N+s}}\,dy \nonumber\\
		&= \text{p.v.}\,\int_{\mathR^n}\frac{\chi_{E}(x) - \chi_{F}(x) - \chi_{E}(y) + \chi_{F}(y)}{|x-y|^{N+s}}\,dy  
	\end{align}
	for any $x \in \partial E \cap \partial F$. Since $E \subset F$, it holds $\chi_E \leq \chi_F$ in $\mathR^n$ and $\chi_E(x)=\chi_F(x)$ for any $x \in \partial E \cap \partial F$. Thus from \eqref{nonlocalMeanCurvature} and the non-negativity of $K$, we obtain the claim.
	
	Therefore, from \eqref{eulerLagrange01},  \eqref{eulerLagrange02}, and the fact that $H^s_{E_{t'}} \geq H^s_{E_{t}}$, we obtain
	\begin{equation}
		t' - f(x_0) \geq t - f(x_0) \nonumber
	\end{equation}
	and it turns out that $t' \geq t$. This contradicts the fact that $t' < t$.

\end{proof}

\subsection{Regularity of boundaries of super-level sets for minimizers}
Now we show some regularity results of the boundary of the set $\{u>t\}$ for each $t$ under suitable assumptions on the datum $f$, where $u$ is a minimizer of $\capF_{K,f}$ with $K(x)=|x|^{-(N+s)}$. From Proposition \ref{comparisonLinfty}, we have that $u \in L^{\infty}(\mathR^n)$ whenever $f \in L^2 \cap L^{\infty}(\mathR^n)$. Since $\{u>t\} = \mathR^n$ if $t < - \|u\|_{L^{\infty}}$ and $\{u>t\} = \emptyset$ if $t \geq \|u\|_{L^{\infty}}$, in the sequel, we focus on the set $\{u>t\}$ only for $t \in [-\|u\|_{L^{\infty}}, \, \|u\|_{L^{\infty}})$ if $f \in L^{\infty}(\mathR^n)$. Recall that, from Corollary \ref{corUniformBoundednessMinimizer}, the super-level set $\{u>t\}$ is bounded uniformly in $t \in [-\|u\|_{L^{\infty}}, \, \|u\|_{L^{\infty}})$.

To obtain our main result on the regularity of minimizers, we exploit the regularity results proved by Caputo and Guillen \cite{CaGu}; Figalli, Fusco, Maggi, Millot, and Morini \cite{FFMMM}; Savin and Valdinoci \cite{SaVa}; and Barrios, Figalli, and Valdinoci \cite{BFV}.

Before recalling the regularity results, we give the definition of ``almost'' minimizers of $P_s$ in the sense of Figalli, et al \cite{FFMMM}. Given $\Lambda>0$, we say that a measurable bounded set $E \subset \mathR^n$ is an {\it almost minimizer} of $P_s$ if
\begin{equation}\label{definitionAlmostMinimizer}
	P_s(E) \leq P_s(F) + \frac{\Lambda}{1-s}|E \Delta F|	
\end{equation} 
for any measurable bounded set $F \subset \mathR^n$. Note that the concept of the almost minimality of $P_s$ was also given by Caputo and Guillen \cite{CaGu} and their definition can include a wider variety of sets than the definition by Figalli, et al. \cite{FFMMM}. In this paper, it is sufficient to apply the definition given by Figalli, et al \cite{FFMMM} and thus we do not write the definition given by Caputo and Guillen \cite{CaGu} here.

First, we recall the regularity of almost minimizers of $P_s$ in the sense of \eqref{definitionAlmostMinimizer}, which was shown by Figalli, Fusco, Maggi, Millot, and Morini \cite[Corollary 3.5]{FFMMM} (see also \cite{CaGu}). This result
is a nonlocal analogue of the theory of Tamanini \cite{Tamanini} on almost minimal surfaces.

\begin{theorem}[\cite{FFMMM}]\label{improvementFlatnessFFMMM}
	If $n \geq 2$, $\Lambda > 0$, and $s_0 \in (0,\,1)$, then there exist positive constants $0 < \varepsilon_0 < 1$, $C_0>0$, and $\alpha < 1$, depending on $n$, $\Lambda$, and $s_0$ only, with the following property: if $E$ is an almost minimizer of $P_s$ with $s \in (s_0, \,1)$ in the sense of \eqref{definitionAlmostMinimizer}, then $\partial E$ is of class $C^{1,\alpha}$ for some $0 < \alpha < 1$ except a closed set of Hausdorff dimension $n-2$.
\end{theorem}


Next we recall the regularity result of fractional minimal cones in $\mathR^2$ by Savin and Valdinoci \cite{SaVa}.
\begin{theorem}[\cite{SaVa}] \label{thmNonlocalCone2d}
	Assume that $E \subset \mathR^2$ is a $s$-fractional minimal cone, namely, $E$ satisfies that $E = t\,E$ for any $t >0$. Then $E$ is a half-plane.
\end{theorem}
In particular, by combining the blow-up and blow-down arguments in \cite{CRS}, one may obtain that $s$-fractional minimal surfaces in $\mathR^2$ are fully $C^{1,\alpha}$-regular for any $\alpha \in (0,\,s)$.
\begin{corollary}[\cite{SaVa}] \label{corNonlocalCone2d}
	If $E$ is an $s$-fractional minimal set in $\Omega \subset \mathR^2$, then $\partial E \cap \Omega'$ is a $C^{1,\alpha}$-curve for any $\Omega' \Subset \Omega$.
\end{corollary}

Originally, the regularity of nonlocal(fractional) minimal surfaces, which are defined by the boundaries of sets minimizing the fractional perimeter, was obtained by Caffarelli, Roquejoffre, and Savin \cite{CRS}. Precisely they proved that every fractional minimal surface is locally $C^{1,\alpha}$ except a closed set of Hausdorff dimension $n-2$. Moreover, thanks to Corollary \ref{corNonlocalCone2d}, this closed set of fractional minimal surfaces has Hausdorff dimension at most $n-3$.

As a consequence of these regularity results, we obtain 
\begin{lemma}[$C^{1,\alpha}$-regularity of boundary of super-level set of minimizers]\label{holderRegualrityLemma}
	Let $s \in (0,\,1)$ and let $f \in L^2 \cap L^{\infty}(\mathR^n)$. Assume that $K(x) = |x|^{-(n+s)}$ for $x \in \mathR^n \setminus \{0\}$ and $u \in BV_K \cap L^2(\mathR^n)$ is a minimizer of the functional $\capF_{K,f}$. Then, for each $t \in \mathR$, the boundary of the super-level set of $u$ is of class $C^{1,\alpha}$ with some $0 < \alpha < 1$, except a closed set of Hausdorff dimension $n-3$. 
\end{lemma}
\begin{proof}
	We fix $t \in \mathR$. Let $x_0 \in \partial \{u>t\}$ and $r>0$ be any number. First, from the assumption on $f$ and Lemma \ref{boundednessMinimizers} in Section \ref{secComparisonMini}, $u$ is non-negative and there exists a constant $R_0>0$ such that $E_t \coloneqq \{u>t\} \subset B_{\frac{R_0}{2}}$ for any $t \geq -\|u\|_{L^{\infty}}$. In order to apply Theorem \ref{improvementFlatnessFFMMM} to our case, it is sufficient to show that each set $E_t$ is an almost minimizer in the sense of \eqref{definitionAlmostMinimizer}. From Lemma \ref{relationMiniTwoEnergies}, we know that $\{u>t\}$ is a solution to the problem
	\begin{equation}\nonumber
		\min\{\capE_{K,f,t}(E) \mid |E| < \infty\}
	\end{equation}
	for each $t \in \mathR$. Hence, from the minimality and boundedness of $E_t$, we have that
	\begin{equation}\label{minimalitySublevel}
		\capE_{K,f,t}(E_t) \leq \capE_{K,f,t}(F)
	\end{equation}
	for any bounded measurable set $F \subset \mathR^n$. Hence, from \eqref{minimalitySublevel}, we can compute as follows: for any bounded measurable set $F$, we have
	\begin{align}\label{estimateQuasiNonlocalMini}
		P_{K}(E_t) - P_{K}(F) &= \capE_{K,f,t}(E_t) - \int_{E_t}(t-f(x))\,dx  \nonumber\\
		&\qquad - \capE_{K,f,t}(F) + \int_{F}(t-f(x)\,dx \nonumber\\
		&\leq \int_{\mathR^n}|\chi_{E_t}- \chi_{F}|\,|t - f(x)|\,dx \nonumber\\
		&\leq \int_{B_r(x_0)}|t - f(x)|\,dx. 
	\end{align}
	Since we assume that $f \in L^{\infty}(\mathR^n)$, we have
	\begin{equation}\label{estimateResidue01}
		\int_{B_r(x_0)}|t - f(x)|\,dx \leq (t +\|f\|_{L^{\infty}(\mathR^n)})\,|E_t \Delta F|.
	\end{equation} 
	Hence, from \eqref{estimateQuasiNonlocalMini} and \eqref{estimateResidue01}, we have
	\begin{equation}\nonumber
		P_K(E_t) \leq  P_K(F) + (t + \|f\|_{L^{\infty}(\mathR^n)})\,|E_t \Delta F|.
	\end{equation}
	Therefore, we apply Theorem \ref{improvementFlatnessFFMMM} and Corollary \ref{corNonlocalCone2d} to conclude that the claim is valid. 
\end{proof}


In addition, we employ another result of the regularity of solutions to integro-differential equations via the bootstrap argument. This result is obtained by Barrios, Figalli, and Valdinoci \cite[Theorem 1.6]{BFV}. They proved the following regularity theorem on the solutions to integro-differential equations. For simplicity, we do not describe the whole statement. See \cite[Theorem 1.6]{BFV} for the full statement. 
\begin{theorem}\label{theoremBootstrapBFV}
	Let $v \in L^{\infty}(\mathR^{n-1})$ be a solution (in the viscosity sense) to the integro-differential equation
	\begin{equation}\nonumber
		\int_{\mathR^{n-1}}A_r(x',\,y')\left( v(x'+y') + v(x'-y') - 2v(x') \right) \,dy' = F(x', v(x'))
	\end{equation}
	for any $x' \in B'_r(0) \subset \mathR^{n-1}$ where $A_r$ satisfies the following assumptions:
	\begin{itemize}
		\item[(A1)] There exist constants $a_0,\,r_0>0$ and $\eta \in (0,\frac{a_0}{4})$ such that 
		\begin{equation*}
			\frac{(1-s)(a_0-\eta)}{|y'|^{n+s}} \leq A_r(x',\, y') \leq \frac{(1-s)(a_0+\eta)}{|y'|^{n+s}} 
		\end{equation*} 
		for any $x' \in B'_r(0)$ and $y' \in B'_{r_0}(0) \setminus \{0\}$.
		
		\item[(A2)] There exists a constant $C_0>0$ such that
		\begin{equation*}
			\| A_r(\cdot,\,y') \|_{C^{0,\beta}(B'_1)} \leq \frac{C_0}{|y'|^{n+s}}
		\end{equation*}
		for any $y' \in B'_{r_0}(0) \setminus \{0\}$.
	\end{itemize}
	and $F \in C^{0,\beta}(B'_r(0))$ with $\beta \in (0,\,1]$. Then $v \in C^{1,s+\alpha}(B'_{\frac{r}{2}}(0))$ for any $\alpha < \beta$.
\end{theorem}

Taking into account all the above arguments, we can obtain that the boundary of the super-level set of the minimizer of $\capF_{K,f}$ has the $C^{2,s+\beta-1}$-regularity under the $\beta$-H\"older regularity of a given datum $f$ with $\beta \in (1-s,\,1]$. Precisely, we prove
\begin{lemma}\label{improvedRegularity}
	Assume that $K(x) = |x|^{-(n+s)}$ for $x \in \mathR^n \setminus \{0\}$ with $s\ in (0,\,1)$ and $f$ is in  $L^2 \cap L^{\infty}(\mathR^n)$. Let $u \in BV_K \cap L^2(\mathR^n)$ be a minimizer of the functional $\capF_{K,f}$. If a datum $f$ is in $C^{0,\beta}_{loc}(\mathR^n)$ with $\beta \in (1-s,\,1]$, then for each $t \in \mathR$, the boundary of the super-level set $\{u > t\}$ is of class $C^{2,s+\alpha-1}$ with $1-s < \alpha < \beta \leq 1$ except a closed set of Hausdorff dimension $n-3$.
\end{lemma}
\begin{proof}
	From Lemma \ref{holderRegualrityLemma} and the assumption that $f \in C^{0,\beta}_{loc} \cap L^{\infty}(\mathR^n)$ with $\beta \in (1-s,\,1]$, the boundary of the set $\{u > t\}$ has full $C^{1,\alpha}$-regularity with some $\alpha \in (0,\,1)$ except a closed set $\Sigma$ of Hausdorff dimension $n-3$, and thus we can represent $\partial \{u>t\} \setminus \Sigma$ locally as a graph of a $C^{1,\alpha}$-function $v_t$ in a bounded domain $U' \subset \mathR^{n-1}$. By employing the computation shown in \cite{BFV}, we may have that $v_t$ satisfies the equation, in the viscosity sense,
	\begin{align}
		&\int_{\mathR^{n-1}} A_r(x',\,y') \frac{v_t(x'+y') + v_t(x'-y') - 2v(x')}{|y'-x'|^{(n-1)+(1+s)}}\,dy' \nonumber\\
		&\quad  = G(x',\,v(x')) + t - f(x',\,v_t(x')) \quad \text{for $x' \in U' \subset \mathR^{n-1}$} \nonumber
	\end{align} 
	where $A_r$ satisfies (A1) and (A2) and $G$ is a smooth function (see \cite{BFV} for the detail). Then, since $f \in C^{0,\beta}_{loc}(\mathR^n)$, we now apply Theorem \ref{theoremBootstrapBFV} several times, if necessary, to conclude that the regularity of $v_t$ can be improved up to $C^{2,s+\alpha-1}$ with $1-s < \alpha < \beta \leq 1$. From the compactness of the boundary of $\{u>t\}$ and by the standard covering argument, we obtain the $C^{2,s+\alpha-1}$-regularity of $\partial \{u>t\}$ for any $\alpha \in (1-s,\,\beta)$. 
\end{proof}

\subsection{Proof of the main regularity result}
By using Lemma \ref{improvedRegularity}, we are now ready to prove the main result of this paper.

Let us briefly explain the strategy of the proof of Theorem \ref{mainTheorem}. Let $t_1, \, t_2 \in [-\|u\|_{L^{\infty}},\,\|u\|_{L^{\infty}})$ with $t_1 < t_2$ and we set $E_1 \coloneqq \{u>t_1\}$ and $E_2 \coloneqq \{u>t_2\}$. Notice that $E_2 \subset E_1$ because $t_1 < t_2$. In order to show the H\"older regularity of $u$, it is sufficient to observe that the boundaries of $E_1$ and $E_2$ are not too close. Precisely, using the regularity of $f \in C^{0,\beta}$, we show the inequality that
\begin{equation}\label{inequalityCrucialMainTheorem}
	t_2 - t_1 \lesssim \left( \dist(\partial E_1, \partial E_2) \right)^{\beta}.
\end{equation}
To see this, we compare the nonlocal mean curvatures on the boundaries $\partial E_1$ and $\partial E_2$. Notice that one can compare the curvatures at points which the boundaries have in common. Thus, we slide $\partial E_1$ (denoted by $\partial E^{\nu}_1$) along the outer unit normal $\nu$ of $\partial E_1$ until $\partial E^{\nu}_1$ touches $\partial E_2$. At the touching point, we can now compare the curvatures between $\partial E^{\nu}_1$ and $\partial E_2$. Moreover, by employing the computation by D\'avila, del Pino, and Wei \cite{DdPW}, we can also compare the curvatures between $\partial E_1$ and $\partial E^{\nu}_1$.  


\begin{proof}[Proof of Theorem \ref{mainTheorem}]
	Let $d_{t} \coloneqq d_{E_t}$ for $t \in [0,\,\infty)$ be a signed distance function from $\partial \{u>t\}$, which is negative inside $\{u>t\}$. We set $E_t \coloneqq \{x\mid u(x)>t\}$ for any $t$. Since $n=2$, from Lemma \ref{improvedRegularity} it follows that all the points on $\partial E_t$ are regular points. Thus, the signed distance function $d_t$ is of class $C^{2,s+\alpha-1}$ in a neighborhood of $\partial E_t$ with $1-s<\alpha<\beta$ (see, for instance, \cite{Wl, DeZo01, DeZo02, Bellettini} for the relation between the distance function and regularity of surfaces).
	
	Recall that, from the assumption on $f$ and Proposition \ref{comparisonLinfty}, we have that $\|u\|_{L^{\infty}} \leq \|f\|_{L^{\infty}} < \infty$. We now take any $t_1 \in [-\|u\|_{L^{\infty}},\,\|u\|_{L^{\infty}})$ and set $E_1 \coloneqq E_{t_1}$. Then we can choose a neighborhood $U_1 \subset \mathR^2$ of the boundary $\partial E_1$ such that $d_1 \coloneqq d_{t_1} \in C^{2,s+\alpha-1}(U_1)$. Moreover, we take any $t_2 \in (-\|u\|_{L^{\infty}},\,\|u\|_{L^{\infty}})$ with $t_2 > t_1$ and set $E_2 \coloneqq E_{t_2}$. Then, from Lemma \ref{corUniformBoundednessMinimizer}, we obtain that there exists a constant $R_c>0$ independent of $t_1$ and $t_2$ such that $E_2 \subset E_1 \subset B_{R_c}$. We can choose points $x_1 \in \partial E_1$ and $x_2 \in \partial E_2$ such that
	\begin{equation}\nonumber
		\tilde{\delta} \coloneqq \dist(\partial E_1,\,\partial E_2) = |x_1 - x_2|.
	\end{equation}
	Since we study the local H\"older regularity of $u$, it is sufficient to consider the case that $x_2 \in U_1$. 
	
	We first show that the following inequality holds: 
	\begin{equation}\label{inequalityHolderRegularity}
		t_2 - t_1 \leq ([f]_{\beta} + C\,\tilde{\delta}^{1-\beta} )\,\tilde{\delta}^{\beta}
	\end{equation}
	where $\beta$ is as in Theorem \ref{mainTheorem} and $C>0$ is a constant depending only on $s$ and $d_1$. 
	
	Without loss of generality, we may assume that $\widetilde{\delta} > 0$. Indeed, if $\widetilde{\delta} = 0$, then, from the definition of $\widetilde{\delta}$, we can easily see that $t_2=t_1$. This implies that the inequality \eqref{inequalityHolderRegularity} is valid. Thus, in the sequel, we always assume that $\tilde{\delta} > 0$.

	
	Now we define $E_1^{\delta}$ as 
	\begin{equation}\nonumber
		E_1^{\delta} \coloneqq \{ x \in E_1 \mid \dist(x,\,\partial E_1) \leq \delta\}
	\end{equation}
	for any $\delta \in (0,\,\tilde{\delta}]$. Then, from the choice of $t_2$ and the definition of $\widetilde{\delta}$, the boundary of $E_1^{\delta}$ can be described as $\partial E_1^{\delta} = \{x-\delta \nabla d_1(x) \mid  x \in \partial E_1 \}$ for any $\delta \in (0,\,\widetilde{\delta}]$, where $\nabla d_1$ coincides with the outer unit normal vector of $\partial E_1$. From the definition of the nonlocal mean curvature, we can easily show that the following comparison inequality holds:
	\begin{equation}\label{comparisonE_1dAndE_2}
		H^K_{E_1^{\tilde{\delta}}}(x_2) \leq H^K_{E_2}(x_2).
	\end{equation} 
	From the choice of $x_1$ and $x_2$, we have $x_2 = x_1 - \delta\,\nabla d_1(x_1)$. Now we compare the two nonlocal curvatures $H^{K}_{E_1^\delta}(x_2)$ and $H^{K}_{E_1}(x_1)$. To do this, we employ the computation shown by D\'avila, del Pino, and Wei in \cite{DdPW} (see also \cite{Cozzi, JuLa}). This computation is on the variation of the nonlocal(fractional) mean curvature. Precisely, we have that, for any set $E \subset \mathR^2$ with a smooth boundary (at least $C^2$), it holds that
	\begin{align}\label{variationNonlocalMC}
		&-\left.\frac{d}{d\delta}\right|_{\delta=0} H^K_{E_{\delta h}}(x-\delta
		h(x)\nabla d_{E}(x)) \nonumber\\
		&= 2\int_{\partial E}\frac{h(y) - h(x)}{|y-x|^{2+s}}\,d\capH^{n-1}(y) \nonumber \\
		&\quad  + 2\int_{\partial E}\frac{(\nabla d_E(y) - \nabla d_E(x)) \cdot \nabla d_E(x)}{|y-x|^{2+s}} \,d\capH^{n-1}(y) 
	\end{align}
	for $x \in \partial E$ where $h \in L^{\infty}(\partial E)$ and $h$ is as smooth as $\partial E$. Here we define $E_{\delta h}$ in such a way that its boundary is given by $\partial E_{\delta h} \coloneqq \{x-\delta\,h(x)\,\nabla d_E(x) \mid x\in \partial E\}$ for any $\delta>0$. Then from \eqref{variationNonlocalMC} and by some computation, we have the estimate of the variation of the nonlocal mean curvature $H^s_{E_1^{\delta}}$ for small $\delta>0$. Precisely we can obtain that there exist constants $C>0$ and $\delta_0>0$, which depends on the space-dimension $n=2$, $s$, and the $L^{\infty}$-norm of $\nabla^2 d_1$ (equivalently the second fundamental form of $\partial E_1$), such that
	\begin{equation}\label{estimateVariationNonlocalMC}
		-\frac{d}{d\delta} H^K_{E_1^{\delta}}(\Psi_{\delta}(x_1)) \leq C\,\int_{\partial E_1}\frac{|\nabla d_1(y) - \nabla d_1(x_1)|^2}{|y-x_1|^{2+s}} \,d\capH^{n-1}(y)
	\end{equation}
	for any $\delta \in (0,\,\delta_0)$ where we set $\Psi_{\delta}(x_1) \coloneqq x-\delta \,\nabla d_1(x)$. Indeed, choosing any smooth cut-off function $\eta_{\varepsilon}$ such that $\spt\eta_{\varepsilon} \subset 
	B^c_{\varepsilon}(0)$, $\eta_{\varepsilon} \equiv 1$ in $B^c_{2\varepsilon}(0)$, and $0 \leq \eta_{\varepsilon} \leq 1$, we can write the nonlocal curvature as follows:
	\begin{align}\label{nonlocalPeriSplit}
		&-H^s_{E_1^{\delta}}(\Psi_{\delta}(x_1)) \nonumber\\
		&= \int_{\mathR^2}\frac{\chi_{E_1^{\delta}}(y) - \chi_{(E_1^{\delta})^c}(y)}{|y-\Psi_{\delta}(x_1)|^{2+s}} \eta_{\varepsilon}(y-\Psi_{\delta}(x_1))\,dy \nonumber\\
		&\qquad + \int_{\mathR^2}\frac{\chi_{E_1^{\delta}}(y) - \chi_{(E_1^{\delta})^c}(y)}{|y-\Psi_{\delta}(x_1)|^{2+s}} (1-\eta_{\varepsilon}(y-\Psi_{\delta}(x_1)))\,dy \nonumber\\
		&\eqqcolon A_{\varepsilon}(\delta) + B_{\varepsilon}(\delta). 
	\end{align}
	Then we can compute the derivative of $A_{\varepsilon}(\delta)$ in \eqref{nonlocalPeriSplit} for small $\delta>0$ in the following manner: setting $\widetilde{y}_{\delta} \coloneqq y-\Psi_{\delta}(x_1)$ for simplicity, we have
	\begin{align}\label{variationNonlocalMC01}
		& \frac{d}{d\delta}\int_{\mathR^2}\frac{\chi_{E_1^{\delta}}(y) - \chi_{(E_1^{\delta})^c}(y)}{|\widetilde{y}_{\delta}|^{2+s}} \eta_{\varepsilon}(\widetilde{y}_{\delta})\,dy \nonumber\\
		&= \int_{\partial E_1^{\delta}} \frac{\eta_{\varepsilon}(\widetilde{y}_{\delta})}{|\widetilde{y}_{\delta}|^{2+s}}  \,d\capH^{n-1}(y) + \int_{\partial (E_1^{\delta})^c} \frac{\eta_{\varepsilon}(\widetilde{y}_{\delta})}{|\widetilde{y}_{\delta}|^{2+s}}  \,d\capH^{n-1}(y) \nonumber\\
		&\qquad - (2+s)\int_{\mathR^2} \frac{\chi_{E_1^{\delta}}(y) - \chi_{(E_1^{\delta})^c}(y)}{|\widetilde{y}_{\delta}|^{4+s}} (y-x_1+ \delta \nabla d_1(x_1)) \cdot \nabla d_1(x_1) \,\eta_{\varepsilon}(\widetilde{y}_{\delta})\,dy \nonumber\\
		&\qquad \quad + \int_{\mathR^2} \frac{\chi_{E_1^{\delta}}(y) - \chi_{(E_1^{\delta})^c}(y)}{|\widetilde{y}_{\delta}|^{2+s}} \nabla \eta_{\varepsilon}(\widetilde{y}_{\delta}) \cdot \nabla d_1(x_1)\,dy 
	\end{align}
	for any $\delta \in (0,\,1)$ with $\Psi_{\delta}(x_1) \in U_1$. Then by using the Gauss-Green theorem, we have
	\begin{align}\label{variationNonlocalMC02}
		&- (2+s)\int_{\mathR^2} \frac{\chi_{E_1^{\delta}}(y) - \chi_{(E_1^{\delta})^c}(y)}{|\widetilde{y}_{\delta}|^{4+s}} (y-x_1+ \delta \nabla d_1(x_1)) \cdot \nabla d_1(x_1) \,\eta_{\varepsilon}(\widetilde{y}_{\delta})\,dy \nonumber\\
		&=  \int_{\mathR^2} (\chi_{E_1^{\delta}}(y) - \chi_{(E_1^{\delta})^c}(y)) \nabla_y \left(\frac{1}{|\widetilde{y}_{\delta}|^{2+s}}\right) \cdot \nabla d_1(x_1) \,\eta_{\varepsilon}(\widetilde{y}_{\delta})\,dy \nonumber\\
		&= \int_{\partial E_1^{\delta}} \frac{\nabla d_1(x_1) \cdot \nabla d_{E_1^{\delta}}(y)}{|\widetilde{y}_{\delta}|^{2+s}} \eta_{\varepsilon}(\widetilde{y}_{\delta}) \,d\capH^{n-1} \nonumber\\
		&\qquad  - \int_{\partial (E_1^{\delta})^c} \frac{\nabla d_1(x_1) \cdot (-\nabla d_{E_1^{\delta}}(y))}{|\widetilde{y}_{\delta}|^{2+s}} \eta_{\varepsilon}(\widetilde{y}_{\delta}) \,d\capH^{n-1} \nonumber\\
		&\qquad \quad - \int_{\mathR^2} \frac{\chi_{E_1^{\delta}}(y) - \chi_{(E_1^{\delta})^c}(y)}{|\widetilde{y}_{\delta}|^{2+s}} \nabla \eta_{\varepsilon}(\widetilde{y}_{\delta}) \cdot \nabla d_1(x_1)\,dy. 
	\end{align}
	Thus from \eqref{variationNonlocalMC01} and \eqref{variationNonlocalMC02}, we obtain
	\begin{align}
		\frac{d}{d\delta}A_{\varepsilon}(\delta)
		&= \int_{\partial E_1^{\delta}} \frac{2- 2(\nabla d_1(x_1) \cdot \nabla d_{E_1^{\delta}}(y))}{|\widetilde{y}_{\delta}|^{2+s}}\eta_{\varepsilon}(\widetilde{y}_{\delta})  \,d\capH^{n-1}(y) \nonumber\\
		&= \int_{\partial E_1^{\delta}} \frac{|\nabla d_1(x_1) - \nabla d_{E_1^{\delta}}(y)|^2}{|\widetilde{y}_{\delta}|^{2+s}}\eta_{\varepsilon}(\widetilde{y}_{\delta})  \,d\capH^{n-1}(y) \nonumber
	\end{align}
	for any small $\delta >0$ with $\Psi_{\delta}(x_1) \in U_1$. Hence from the change of variables, we obtain 
	\begin{align}
		\frac{d}{d\delta}A_{\varepsilon}(\delta)
		&= \int_{\partial E_1} \frac{|\nabla d_1(x_1) - \nabla d_{1}(y)|^2}{|\Psi_{\delta}(y)-\Psi_{\delta}(x_1)|^{2+s}}\eta_{\varepsilon}(\Psi_{\delta}(y)-\Psi_{\delta}(x_1)) \,J_{\partial E_1}\Psi_{\delta}(y)\,d\capH^{n-1}(y) \nonumber
	\end{align}
	where $J_{\partial E_1}\Psi_{\delta}(y)$ is the tangential Jacobian of $\partial E_1$ at $y$. As is shown in \cite{DdPW}, we can have that there exist constants $c'>0$ and $\delta'>0$, depending on the space-dimension $n=2$ and $s$ but independent of $\varepsilon >0$, such that $|\frac{d}{d\delta}B_{\varepsilon}(\delta)| \leq c'\varepsilon^{1-s}$ for any $\delta \in (0,\,\delta')$ and $\varepsilon \in (0,\,1)$. Therefore, we conclude that
	\begin{align}
		-\frac{d}{d\delta}H^s_{E_1^{\delta}}(\Psi_{\delta}(x_1)) &= \lim_{\varepsilon \downarrow 0}\left(\frac{d}{d\delta}A_{\varepsilon}(\delta) + \frac{d}{d\delta} B_{\varepsilon}(\delta) \right) \nonumber\\
		&= \int_{\partial E_1} \frac{|\nabla d_1(x_1) - \nabla d_{1}(y)|^2}{|\Psi_{\delta}(y)-\Psi_{\delta}(x_1)|^{2+s}} J_{\partial E_1}\Psi_{\delta}(y) \,d\capH^{n-1}(y) \nonumber
	\end{align}
	for any $\delta \in (0,\,\delta'_0)$ where $\delta'_0>0$ is a constant depending on the space-dimension $n=2$, $s$, and the $L^{\infty}$-norm of $\nabla^2 d_1$. From the definition of $\Psi_{\delta}$, we have that there exists a constant $C_0>0$ depending on the space-dimension $n=2$, $s$, and the $L^{\infty}$-norm of $\nabla^2 d_1$, such that 
	\begin{equation}\nonumber
		\frac{J_{\partial E_1}\Psi_{\delta}(y)}{|\Psi_{\delta}(y) -\Psi_{\delta}(x_1)|^{2+s}} \leq  \frac{C_0}{|y-x_1|^{2+s}}
	\end{equation} 
	for any $y \in \partial E_1$ and $\delta \in (0,\,\delta'_0)$. Therefore we obtain that there exist constants $C>0$ and $\delta_0>0$, depending on the space-dimension $n=2$, $s$, and the second derivative of $d_1$ but independent of $\delta$, such that the inequality \eqref{estimateVariationNonlocalMC} with the constant $C$ holds for any $\delta \in (0,\,\delta_0)$. Thus, from the fundamental theorem of calculus and \eqref{estimateVariationNonlocalMC}, we obtain that
	\begin{align}\label{estimateNonlocalCurvatures}
		&-H^K_{E_1^{\delta}}(x-\delta
		\,\nabla d_1(x)) \nonumber\\
		&= -H^K_{E_1}(x_1) - \delta\,\int_{0}^{1} \frac{d}{d\delta} H^K_{E_1^{\delta}}(x- \lambda\delta
		\,\nabla d_1(x)) \,d\lambda \nonumber\\
		&\leq -H^K_{E_1}(x_1) + C\,\delta\,\int_{\partial E_1}\frac{|\nabla d_1(y) - \nabla d_1(x_1)|^2}{|y-x_1|^{2+s}} \,d\capH^{n-1}(y) 
	\end{align}
	for any $\delta \in (0,\,\delta_0)$. Now we show that the integral 
	\begin{equation}\nonumber
		\int_{\partial E_1}\frac{|\nabla d_1(y) - \nabla d_1(x_1)|^2}{|y-x_1|^{2+s}} \,d\capH^{n-1}(y)
	\end{equation}
	is uniformly bounded for any $x_1 \in V$ and any open set $V \subsetneq U_1$. Indeed, we define the set $U^r_1 \coloneqq \{x \in U_1 \mid \dist(x,\,\partial U_1)>r\}$ for any $r>0$ satisfying that $B_{2r}(x) \subset U_1$ for any $x \in U_1$. Then we can compute the integral as follows: for any $x_1 \in U^r_1$, it holds that 
	\begin{align}\label{estimateFracNormNormalVec}
		&\int_{\partial E_1}\frac{|\nabla d_1(y) - \nabla d_1(x_1)|^2}{|y-x_1|^{2+s}} \,d\capH^{n-1}(y) \nonumber\\
		&= \int_{\partial E_1 \cap B_r(x_1)}\frac{|\nabla d_1(y) - \nabla d_1(x_1)|^2}{|y-x_1|^{2+s}} \,d\capH^{n-1}(y) \nonumber\\
		&\qquad + \int_{\partial E_1 \cap B^c_r(x_1)}\frac{|\nabla d_1(y) - \nabla d_1(x_1)|^2}{|y-x_1|^{2+s}} \,d\capH^{n-1}(y) \nonumber\\
		&\leq \int_{\partial E_1 \cap B_r(x_1)}\frac{|\nabla d_1(y) - \nabla d_1(x_1)|^2}{|y-x_1|^2} \frac{1}{|y-x_1|^{n-2+s}} \,d\capH^{n-1}(y) \nonumber\\
		&\qquad + \int_{\partial E_1 \cap B^c_r(x_1)}\frac{4}{|y-x_1|^{2+s}} \,d\capH^{n-1}(y). 
	\end{align}
	From the fundamental theorem of calculus and the fact that $B_r(x_1) \subset U_1$ for any $x_1 \in U^r_1$, we have that
	\begin{equation}\label{estiGradientDistance}
		\frac{|\nabla d_1(y) - \nabla d_1(x_1)|^2}{|y-x_1|^2} \leq \|\nabla^2 d_1\|_{L^{\infty}(B_r(x_1))}^2
	\end{equation}
	for any $y \in B_r(x_1)$. Thus from \eqref{estimateFracNormNormalVec} and \eqref{estiGradientDistance} and noticing that $x_1 \in U^r_1$ and $E_t \subset B_{R_c}$ holds uniformly in $t \geq c$ where $c \coloneqq - \|u\|_{L^{\infty}} > -\infty$, we obtain
	\begin{align}\label{estimateGradDistanceonSurface}
		\int_{\partial E_1}\frac{|\nabla d_1(y) - \nabla d_1(x_1)|^2}{|y-x_1|^{2+s}} \,d\capH^{n-1}(y) &\leq  c_1\, \|\nabla^2 d_1\|_{L^{\infty}(B_r(x_1))}^3\,r^{1-s} \nonumber\\
		&\qquad + \frac{c_2\,\|\nabla^2 d_1\|_{L^{\infty}(U_1)}}{r^s} 
	\end{align}
	where $c_1>0$ and $c_2>0$ are constants depending on the space-dimension $n=2$ and $s$. Since we choose any $r$ in such a way that $B_r(x_1) \subset U_1$, we conclude the claim is valid. Thus, from \eqref{estimateNonlocalCurvatures} and \eqref{estimateGradDistanceonSurface}, we finally obtain the inequality
	\begin{equation}\label{comparisonNonlocalMCSmallDiff}
		-H^K_{E_1^{\delta}}(x_1-\delta
		\,\nabla d_1(x)) \leq -H^K_{E_1}(x_1) + C(n,s,R_c)\,\delta
	\end{equation}
	for any $\delta \in (0,\,\delta_0)$ where $C(n,s,R_c)>0$ ($n=2$ is the space-dimension) and $\delta_0>0$ are some constants, which also depend on the $L^{\infty}$-norm of $\nabla^2 d_1$. Note that the constant $\delta_0$ can be bounded by the inverse of the $L^{\infty}$-norm of $\nabla^2 d_1$. Thus from \eqref{comparisonE_1dAndE_2} and \eqref{comparisonNonlocalMCSmallDiff}, we have that, for any $\delta \in (0,\,\delta_0)$, 
	\begin{equation}\label{comparisonNonlocalCurvature}
		-H^K_{E_2}(x_2) \leq -H^K_{E_1}(x_1) + C(n,s,R_c)\,\delta.
	\end{equation}
	Now we consider the following two cases:
	
	\textit{Case 1}: $0< \tilde{\delta} < \delta_0$. In this case, we simply substitute $\delta = \tilde{\delta}$ with \eqref{comparisonNonlocalCurvature} and obtain
	\begin{equation}\nonumber
		-H^K_{E_2}(x_2) \leq -H^K_{E_1}(x_1) + C(n,s,R_c)\,\tilde{\delta}
	\end{equation}
	where $\tilde{\delta} = \dist(\partial E_1,\,\partial E_2)$.
	
	\textit{Case 2}: $\tilde{\delta} \geq \delta_0$. In this case, there exists a number $N \in \mathN$ such that $ \frac{\tilde{\delta}}{N} < \|\nabla^2 d_1\|_{L^{\infty}(U_1)}^{-1}$. Then setting $\tilde{\delta}_k \coloneqq \frac{k}{N}\tilde{\delta}$ for each $k \in \{1,\cdots,\,N\}$ and taking into account all the above arguments, we obtain the inequality that
	\begin{equation}\label{comparisonNonlocalCurvatureIteration}
		-H^{K}_{E_1^{\tilde{\delta}_k}}(x^{\tilde{\delta}_k}_1) \leq -H^{K}_{E_1^{\tilde{\delta}_{k-1}}}(x^{\tilde{\delta}_{k-1}}_1) + C(n,s,R_c)\,\frac{\tilde{\delta}}{N}
	\end{equation}
	for each $k \in \{1,\cdots,\,N\}$ where we understand the notation $x^{\tilde{\delta}_0}_1 = x_1$ and $E_1^{\tilde{\delta}_0} = E_1$. Thus by summing the inequality \eqref{comparisonNonlocalCurvatureIteration} for all $i \in \{1,\cdots,\,N\}$, we obtain
	\begin{align}
		-H^{K}_{E_1^{\tilde{\delta}}}(x_2) &=  -H^{K}_{E_1^{\tilde{\delta}_N}}(x^{\tilde{\delta}_N}_1) \nonumber\\
		&\leq  -H^{K}_{E_1^{\tilde{\delta}_0}}(x^{\tilde{\delta}_0}_1) + N\,C(n,s,R_c)\,\frac{\tilde{\delta}}{N} = -H^{K}_{E_1}(x_1) + C(n,s,R_c)\,\tilde{\delta} \nonumber
	\end{align}
	where $\tilde{\delta} = \dist(\partial E_1,\,\partial E_2)$. In both cases, we finally obtain the inequality
	\begin{equation}\label{comparisonFractionalMeanCurvatures}
		- H^K_{E_2}(x_2) \leq -H^{K}_{E_1}(x_1) + C(n,s,R_c)\,\tilde{\delta}.
	\end{equation}
	Thanks to Lemma \ref{improvedRegularity}, the Euler-Lagrange equation
	\begin{equation}
		H^s_{E_t}(x) + t - f(x) = 0
	\end{equation}
	holds for every $x \in \partial E_t$. Then, since $E_i$ is the minimizer of $\capE_{K,f,t_i}$ for $i \in \{1,2\}$ and from \eqref{comparisonFractionalMeanCurvatures}, we obtain
	\begin{equation}\nonumber
		t_2 - t_1 \leq f(x_2) - f(x_1) + C(n,s,R_c)\,\tilde{\delta}.
	\end{equation}
	Recalling the definition of $x_2$, the H\"older continuity of $f$, and the fact that $E_t \subset B_{R_c}$ for any $t \geq c$, we conclude that
	\begin{equation}\label{keyInequalityHolder}
		t_2 - t_1 \leq ([f]_{\beta}(B_{R_c}) + C(n,s,R_c)\,\tilde{\delta}^{1-\beta})\,\tilde{\delta}^\beta
	\end{equation}
	where $[f]_{\beta}(B_{R_c})$ is the H\"older constant of $f$ in $B_{R_c}$ given as
	\begin{equation}\nonumber
		[f]_{\beta}(B_{R_c}) \coloneqq \sup_{x,\,y \in B_{R_c}, \, x \neq y}\frac{|f(x) - f(y)|}{|x-y|^{\beta}}
	\end{equation}
	and the constant $\tilde{\delta}$ is defined as $\tilde{\delta} \coloneqq \dist(\partial E_1,\,\partial E_2)$. Note that the constant $C(n,s,R_c)>0$ also depends on the $L^{\infty}$-norm of $\nabla^2 d_1$.
	
	We are now ready to prove the local H\"older continuity of $u$. Let $B_{r_0}(x_0) \subset \mathR^2$ be any open ball of radius $r_0$ with $x_0 \in \{u = t_0 \}$ for a number $t_0 \geq c \coloneqq -\|u\|_{L^{\infty}}$. We take any points $x,\,y \in B_{r_0}(x_0)$ with $x \neq y$ and set $t_1,\,t_2 \in \mathR$ as $t_1 \coloneqq u(x)$ and $t_2 \coloneqq u(y)$. We may assume that $t_1 > t_2 \geq c$ because we only repeat the same argument in the case of $t_1 < t_2$. In addition to this, we also assume that $t_1 > t_0 > t_2$. Indeed, in the case of $t_1 > t_2 \geq t_0$ or $t_0 \geq t_1 > t_2$, it is sufficient to take another point $x_0' \in B_{r_0}(x_0)$ and $t_0' \in \mathR$ such that $x_0' \in \{ u = t_0' \}$ and $t_1 > t_0' > t_2$, and do the argument that we will show below. Moreover, since we only observe the local regularity of $u$, it is sufficient to consider the case that $B_{r_0}(x_0) \subset U_0$ where $U_0$ is a neighborhood of $\partial \{u > t_0\}$ such that the signed distance function from $\partial \{u > t_0\}$ is of class $C^{2, s+\alpha-1}(U_0)$ with $\alpha \in (1-s,\, 1)$. Indeed, if $x \in B_{r_0}(x_0) \setminus U_0$ and $y \in B_{r_0}(x_0)$, then, from the continuity of $u$, we can choose a point $z_0$ in $B_{r_0}(x_0)$ and close to $x$ such that the estimate $|u(x) - u(z_0)| \leq |x-y|^{\beta}$ holds and $t_1 = u(x) > u(z_0) \geq u(y) = t_2$. In the case of $z_0 \in U_0$, we just apply the argument that we will show below with \eqref{keyInequalityHolder} for $z_0$, $x_0$, and $y$; otherwise we can repeat the above argument until we have the point belonging to $U_0$.
	
	Now we choose sufficiently small $\varepsilon > 0$ such that $t_1 - \varepsilon > t_0$ and $t_0 - \varepsilon > t_2$ and then we have that $x \in \{ u > t_1 -\varepsilon\}$, $y\in \{ u > t_2 - \varepsilon\}$, and $x_0 \in \{ u > t_0 - \varepsilon\}$. Hence, from \eqref{keyInequalityHolder} and the fact that $x,\,y \in B_{r_0}(x_0)$, we obtain the two inequalities
	\begin{align}\label{estimateHolder01}
		u(x) - u(x_0) = t_1 - \varepsilon - (t_0 - \varepsilon) &\leq ( [f]_{\beta}(B_{R_c}) + C(n,s,R_c)\,\tilde{\delta}_1^{1-\beta})\,\tilde{\delta}_1^{\beta} \nonumber\\
		&\leq ( [f]_{\beta}(B_{R_c}) + C(n,s,R_c)\,r_0^{1-\beta})\,\tilde{\delta}_1^{\beta}. 
	\end{align}
	and 
	\begin{align}\label{estimateHolder02}
		u(x_0) - u(y) = t_0 - \varepsilon - (t_2 - \varepsilon) &\leq ( [f]_{\beta}(B_{R_c}) + C(n,s,R_c)\,\tilde{\delta}_2^{1-\beta})\,\tilde{\delta}_2^{\beta} \nonumber\\
		&\leq ( [f]_{\beta}(B_{R_c}) + C(n,s,R_c)\,r_0^{1-\beta})\,\tilde{\delta}_2^{\beta} 
	\end{align}
	where we set $\tilde{\delta}_1 \coloneqq \dist(\partial E_{t_0},\partial E_{t_1})$ and $\tilde{\delta}_2 \coloneqq \dist(\partial E_{t_0},\partial E_{t_2})$. Note that the constant $C(n,s,R_c)>0$ also depends on the $L^{\infty}$-norm of $\nabla^2 d_{t_0}$, which can be uniformly bounded in $B_{r_0}(x_0)$. Notice that the inequality
	\begin{equation}\nonumber
		\tilde{\delta}_1 + \tilde{\delta}_2 = \dist(\partial E_{t_0},\,\partial E_{t_1}) + \dist(\partial E_{t_0},\,\partial E_{t_2}) \leq \dist(\partial E_{t_1},\,\partial E_{t_2}) \leq |x-y|
	\end{equation}
	holds because of the fact that $E_{t_1} \subset E_{t_0} \subset E_{t_2}$. Therefore from \eqref{estimateHolder01} and \eqref{estimateHolder02}, we obtain that there exists a constant $C=C(n,s,f,R_c,r_0,x_0)>0$ (we have assumed that the space-dimension $n$ is two) such that
	\begin{align}
		|u(x) - u(y)| &= |u(x) - u(x_0) + u(x_0) - u(y)|  \nonumber\\
		&\leq C\,(\tilde{\delta}_1^{\beta} + \tilde{\delta}_2^{\beta}) \leq C\,2^{1-\beta}(\tilde{\delta}_1 + \tilde{\delta}_2)^{\beta} \leq 2^{1-\beta}C\, |x-y|^{\beta}. \nonumber
	\end{align}
	Here, in the second inequality, we have used the fact that $2^{1-\beta}(x+1)^{\beta} \geq x^{\beta} + 1$ for any $x \geq 1$ and $\beta \in (0,\,1)$ and applied this fact with $x=\widetilde{\delta}_1\,\widetilde{\delta}_2^{-1}$ if $\widetilde{\delta}_1 \geq \widetilde{\delta}_2$ or $x=\widetilde{\delta}_2\,\widetilde{\delta}_1^{-1}$ if $\widetilde{\delta}_1 < \widetilde{\delta}_1$.
\end{proof}


\end{document}